\documentclass[12pt]{amsart}
\usepackage{amsmath, amssymb, amsthm, latexsym}
\usepackage{amssymb}
\usepackage{latexsym}
\usepackage[center]{caption}
\usepackage{tikz}
\newcounter{braid}
\newcounter{strands}

\pgfkeyssetvalue{/tikz/braid height}{1cm}
\pgfkeyssetvalue{/tikz/braid width}{1cm}
\pgfkeyssetvalue{/tikz/braid start}{(0,0)}
\pgfkeyssetvalue{/tikz/braid colour}{black}
\pgfkeys{/tikz/strands/.code={\setcounter{strands}{#1}}}

\makeatletter
\def\cross{%
  \@ifnextchar^{\message{Got sup}\cross@sup}{\cross@sub}}

\def\cross@sup^#1_#2{\render@cross{#2}{#1}}

\def\cross@sub_#1{\@ifnextchar^{\cross@@sub{#1}}{\render@cross{#1}{1}}}

\def\cross@@sub#1^#2{\render@cross{#1}{#2}}

\def\render@cross#1#2{
  \def\strand{#1}
  \def\crossing{#2}
  \pgfmathsetmacro{\cross@y}{-\value{braid}*\braid@h}
  \pgfmathtruncatemacro{\nextstrand}{#1+1}
  \foreach \thread in {1,...,\value{strands}}
  {
    \pgfmathsetmacro{\strand@x}{\thread * \braid@w}
    \ifnum\thread=\strand
    \pgfmathsetmacro{\over@x}{\strand * \braid@w + .5*(1 - \crossing) * \braid@w}
    \pgfmathsetmacro{\under@x}{\strand * \braid@w + .5*(1 + \crossing) * \braid@w}
    \draw[braid] \pgfkeysvalueof{/tikz/braid start} +(\under@x pt,\cross@y pt) to[out=-90,in=90] +(\over@x pt,\cross@y pt -\braid@h);
    \draw[braid] \pgfkeysvalueof{/tikz/braid start} +(\over@x pt,\cross@y pt) to[out=-90,in=90] +(\under@x pt,\cross@y pt -\braid@h);
    \else
    \ifnum\thread=\nextstrand
    \else
     \draw[braid] \pgfkeysvalueof{/tikz/braid start} ++(\strand@x pt,\cross@y pt) -- ++(0,-\braid@h);
    \fi
   \fi
  }
  \stepcounter{braid}
}

\tikzset{braid/.style={double=\pgfkeysvalueof{/tikz/braid colour},double distance=1pt,line width=2pt,white}}

\newcommand{\braid}[2][]{%
  \begingroup
  \pgfkeys{/tikz/strands=2}
  \tikzset{#1}
  \pgfkeysgetvalue{/tikz/braid width}{\braid@w}
  \pgfkeysgetvalue{/tikz/braid height}{\braid@h}
  \setcounter{braid}{0}
  \let\sigma=\cross
  #2
  \endgroup
}
\makeatother

\input xypic
\newtheorem{theorem}{Theorem}
\newtheorem{proposition}[theorem]{Proposition}

\newtheorem{lemma}[theorem]{Lemma}


\def\Z{\mathbb{Z}}

\def\B{\mathbb{B}}
\def\C{\mathbb{C}}

\def\R{\mathbb{R}}
\def\C{\mathbb{C}}

\def\N{\mathbb{N}}

\def\Pi{\mathbb{P}^{\infty}}

\def\qed{\hfill$\square$\medskip}

\def\Zpk{\mathbb{Z}/p^{k}}
\def\Zpk1{\mathbb{Z}/p^{k-1}}

\newcommand{\rref}[1]{(\ref{#1})}

\newcommand{\beg}[2]{\begin{equation}\label{#1}#2\end{equation}}
\def\r{\rightarrow}

\def\sl2{\widetilde{SL_{2}(\Z)}}

\title[Definition and realization of a modular functor]{On the definition \\and K-theory realization \\of a modular functor}
\author{Igor Kriz and Luhang Lai}
\thanks{The first author was supported by NSF grant DMS 1102614.
The second author was supported by a research fellowship from the University of Michigan.}

\begin{document}

\maketitle

\begin{abstract}
We present a definition of a (super)-modular functor which includes certain interesting cases
that previous definitions do not allow. We also introduce a notion of topological twisting
of a modular functor, and construct formally a realization by a 2-dimensional topological field
theory valued in twisted K-modules. We discuss, among other things, the $N=1$-supersymmetric
minimal models from the point of view of this formalism.
\end{abstract}

\section{Introduction}\label{intro}

This paper is being written on the 10th anniversary of the publication of the first author's
paper \cite{spin}, and it grew out of a project of writing a sequel of \cite{spin},
solving certain questions posed there, and also commenting on how the material
relates to some topics of current interest. During the process of writing the paper,
the authors received many comments asking
for examples of their theory. Supplemented with those examples, which required venturing
into other fields of mathematics, the scope of the paper
now exceeds by far what was originally intended.

\vspace{3mm}
The original aim of \cite{spin} was to work out 
the modularity behavior of certain series known in mathematical
physics as partition functions of chiral conformal field theories
with $1$-dimensional conformal anomaly. The author intended to use the outline of
Segal \cite{segal}, and work things out in more mathematical detail. 
However, unexpected complications arose. In correspondence with
P. Deligne \cite{deligne} (see also \cite{deligne1}), the author learned that an
important example of 1-dimensional conformal anomaly, known as the {\em Quillen determinant}
of a Riemann surface,
only satisfies the desired gluing axioms when considered as a 
{\em super-line} instead of just an ordinary complex line. A super-line is the same
thing as a line, but with an additional bit of information, labelling it as either even or odd.
In the coherence isomorphism $L\otimes M\cong M\otimes L$,
if both $L$ and $M$ are odd, a sign of $-1$ is inserted. 
The fact that the determinant is a super-line and not a line has a subtle and profound implication on  
its partition function: the partition function is, in fact, $0$. This phenomenon was
well familiar in mathematical physics, but was somewhat subtle to capture rigorously. 

\vspace{3mm}
The situation 
is even more convoluted in the case of another conformal field theory,
known as the chiral fermion theory of central charge $c=1/2$
on Riemann surfaces with spin structure. While the
conformal anomaly is still ``1-dimensional'' (meaning ``invertible''), it cannot be given consistent signs even when we
use super-lines (\cite{deligne}). The odd spinors turn out to require the use of the non-trivial
element of the super-Brauer group $sBr(\C)\cong \Z/2$, and a consistent theory can be built using
the 2-category $\mathcal{D}_0$ of Clifford algebras over $\C$, graded Morita equivalences and
degree 0 isomorphisms. 

\vspace{3mm}
In \cite{spin}, a rigorous concept 1-dimensional (or invertible) modular functor was introduced
which included the situations mentioned above. In fact, such structures were
classified. An important question left open was how to generalize the concept beyond
the 1-dimensional case. 
Segal \cite{segal} previously outlined a definition of a modular functor
(and coined the term),
but did not discuss the super- or Clifford cases, or, in fact, the coherence diagrams required. 
Examples from mathematical physics are, at least conjecturally, abundant
(just to begin, see \cite{mini, wzw, moore, fermion}). Moreover, super- examples are also abundant,
because many of the examples have supersymmetry (such as
super-minimal models \cite{supermini} or the super-WZW model \cite{swzw}), and supersymmetry
requires super-modular functors. (In this paper, we discuss modular functors which are super-, but
not the stronger condition of super-symmetry, since that requires a rigorous theory of 
super-Riemann surfaces, which still has not been worked out in mathematical detail, cf. \cite{cr}.)
While a number of approaches to rigorous definitions of modular functors and similar structures
were proposed \cite{bk, rt, huang, hk, kf}, the question of the correct definition of 
a super-modular functor, and a Clifford modular functor, remained open. It is solved in
the present paper.

\vspace{3mm}
There were several recent developments which raised interest in this topic, and
related it to a field of homotopy theory, namely K-theory. In \cite{hkk}, D.Kriz, I.Kriz and P.Hu
were considering a generalized (co)-homology version of Khovanov homology,
an invariant in knot theory which is a ``categorification'' of the Jones polynomial. Using
topological field theory methods, \cite{hkk} constructs a stable homotopy version
of Khovanov homology; the construction is an alternative to a previous construction of Lipshitz and
Sarkar \cite{ls}, which used Morse theory. Invariance under the Reidemeister moves used for proving knot
invariance only requires an {\em embedded topological field theory}, which is what makes
a stable homotopy type realization possible. In \cite{hkk},  a more canonical K-theoretical version
is also discussed, which uses a more complete topological quantum field theory; this version is related to 
modular functors. 

\vspace{3mm}
From the point of view of \cite{hkk}, a topological modular functor is, roughly speaking, 
a 2-dimensional topological quantum field theory
valued in 2-vector spaces. The {\em K-theory spectrum} $K$ is an object of stable homotopy theory
analogous to a commutative ring (called an $E_\infty$ ring spectrum). For an $E_\infty$ ring spectrum,
there is a concept of a module. The K-theory realization of
\cite{hkk} constructs, from a topological modular functor, a 2-dimensional topological quantum
field theory valued in $K$-modules. The passage to stable homotopy theory
requires a sophisticated device called {\em multiplicative infinite loop space machine}. While
many versions of such a machine are known to algebraic topologists,
a version which seemed flexible enough for discussing topological field theory was discovered only in 2006
by Elmendorf and Mandell \cite{em}; it used the concept of a {\em multicategory}.
Because of this, a definition of topological  modular functors given in \cite{hkk} uses
multicategories. 

\vspace{3mm}
In the present paper, we give a generalization of the definition from
\cite{hkk} to super- and Clifford modular functors, and construct
their {\em $K$-module realization} extending the construction of \cite{hkk}. 
Mathematical physics, in fact, makes a K-theoretical realization of
modular functors desirable also because of the fact that modular functors also classify ``Cardy branes'' \cite{cardy}, which
correspond to certain ``boundary sectors'' of conformal field theories.
Witten \cite{witten} argued that branes should be classified by $K$-theory, which, 
in mathematics, therefore includes a K-theory realization of modular functors. 

\vspace{3mm}
There was still another clue directly related to the super- and Clifford case: it was noticed \cite{aseg,fht}
that the double bar construction 
(i.e. on both 2-morphisms and 1-morphisms)
of the Deligne 2-category $\mathcal{D}_0$ discussed above gives the classifying
space of geometrical twistings of the $E_\infty$ ring spectrum $K$. In this paper,
we construct this twisting space
{\em as an $E_\infty$-space}. This observation is, in fact, essentially equivalent to constructing
K-theory realizations of the invertible super- and Clifford modular functors discussed in 
\cite{spin}. But it also points to the need of a model of the $E_\infty$ ring spectrum $K$ which would
handle all the elements of the ``Picard group'' $Pic(K)=sBr(\C)$. Specifically, by Bott periodicity, $K$-theory is invariant
under dimensional shift (suspension) by $2$.
The non-trivial element of $sBr(\C)$
should be realized by a single suspension of $K$, and the second ``tensor power'' of this element should
give $K$ again: this phenomenon does not arise, for example, in algebraic K-theory (where Bott periodicity
involves a Tate twist), and hence cannot
be captured by any ``finite'', or algebraic, construction. Fortunately, a model
of K-theory which handles this case was discovered by Atiyah and Singer \cite{as}, (it was also
used in \cite{aseg,fht}). Making the Atiyah-Singer model work in the context of $E_\infty$ ring spectra
requires some additional work, which is also treated in the present paper.

\vspace{3mm}
The paper \cite{fht}
identifies the Verlinde algebra of the WZW model as the equivariant twisted K-theory of
a (say, simply connected) compact Lie group $G$ acting on itself by conjugation. 
It therefore begs the question whether the fixed point spectrum  of the twisted K-theory spectrum
$K_{G,\tau}(G)$ itself is the K-theory realization of the modular functor of the WZW model. 
We address a slightly weaker question here: we show that the twisted $K$-theory realization
depends only on a weaker structure of a {\em projective} modular functor. Using
the machinery of modular tensor categories (as described, say, in \cite{bk}), 
we are able to construct the twisted $K$-theory data from the WZW model, although
a more direct connection would still be desirable. 

In particular, it is important to note tha the twisting comes from the fact that the modular functors
of the WZW models are not topological, but holomorphic. The violation
of topological invariance in a holomorphic modular functor is expressed by a single
numerical invariant called {\em central charge}. In this paper, we extract the
topological information in the central charge in an invariant we call {\em topological
twisting}. It turns out to be a {\em torsion} invariant (it vanishes when the central charge
is divisible by 4). We show that holomorphic modular functors
can be realized as 2-dimensional quantum field theories valued in twisted K-modules,
which is precisely the type of structure present in the Freed-Hopkins-Teleman
spectrum $(K_{G,\tau}(G))^G$. This is interesting because of
speculations about possible connections of 1-dimensional modular functors
with dimensions of elliptic cohomology \cite{bdh,st}: in connection with
those ideas, a torsion
topological invariant associated with central charge was previously conjectured.

\vspace{3mm}
In view of these observations, we felt we should construct non-trivial examples at least of 
{\em projective} modular functors which would utilize the full Clifford generality we introduce.
As a somewhat typical case, we discuss the example of the $N=1$-supersymmetric minimal
models. Getting projective Clifford modular functors
from these examples rigorously is a lot of work. First of all, we
need to identify the additional data on modular tensor categories which allow us to
produce a Clifford modular functor. We treat this in Section \ref{sexamp}, where we
introduce the concept of a {\em Clifford modular tensor category}, and show
how to use it to construct a projective Clifford modular functor. This
construction is general. It also applies to other examples which will be discussed elsewhere.
Then, in Section \ref{svoa}, we gather the relevant facts
from the vertex algebra literature and verify that they fit the formalism. 
In the end, we do associate a projective Clifford modular functor, and hence
a twisted $K$-theory realization, with the $N=1$-supersymmetric minimal model
super-vertex operator algebras.

\vspace{3mm}
The present paper is organized as follows: Multicategories are discussed in Section \ref{multi}.
A definition of a classical (not super- or Clifford) modular functor in the sense of \cite{segal} with 
multicategories is given in Section \ref{smf}. The super- and Clifford version of a modular
functor, which is the main definition
of the present paper, is given in Section \ref{real}. The K-theory realization of topological
modular functors is presented in Section \ref{singer}. The K-theory realization of a
general modular functor, with a discussion of central charge, is given in Section \ref{twist}.
Section \ref{sexamp} contains the concept of a Clifford modular tensor category,
which is a refinement of a modular tensor category which produces a projective Clifford
modular functor. In Section \ref{svoa}, we discuss how this applies to the case of super
vertex algebras, and we discuss concretely the example of the $N=1$-supersymmetric
minimal models.
Section \ref{sapp} is an Appendix, which contains some techincal results 
we needed for technical reasons to construct realization. Notably, this
includes May-Thomason rectification, and a topological version of the Joyal-Street construction.
We also discuss singular vectors in Verma modules over the $N=1$ NS algebra.

\vspace{3mm}

\noindent
{\bf Acknowledgements:} The authors are indebted to Victor Kac for discussions on vertex
algebras, and to Jan Nekov\'{a}\v{r} for additional historical
comments on the Quillen determinant.

\vspace{3mm}

\section{Multicategories}\label{multi}

In this paper, we study among other things the interface between topological field theories
and stable homotopy theory. This requires what is known as {\em multiplicative infinite
loop space theory}, which is a difficult subject to treat rigorously. As in \cite{hkk}, we use the 
approach of Elmendorf and Mandell, based on {\em multicategories}. 

A multicategory \cite{em} is the same thing as a multisorted operad (also known as a colored operad). This means
we are given a set $S$ and for each $n=0,1,2,\dots$ a set $\mathcal{C}(n)$ with a map to
$S^{\times n}\times S$ with a $\Sigma_n$-action fibered over the $\Sigma_n$-action on $S^n$ by
permutation of factors, a composition
$$\mathcal{C}(n)\times_{S^n} (\mathcal{C}(k_1)\times\dots\times \mathcal{C}(k_n)\r
\mathcal{C}(k_1+\dots+k_n)$$
over the identity on $S^{k_1+\dots +k_n}\times S$ and a unit
$$S\r\mathcal{C}(1)$$
over the diagonal map $S\r S\times S$.
These operations are subject to the same commutative diagram as required in the definition of
an operad \cite{operad}.

In a multicategory, we refer to $S$ as the set of {\em objects} and to elements of $\mathcal{C}(n)$
over $(s_1,\dots,s_n,s)$, $s,s_i\in S$, as {\em multimorphisms}
\beg{ee+}{(s_1,\dots,s_n)\r s.}
The set of multimorphisms \rref{ee+} will be denoted by $\mathcal{C}(s_1,\dots,s_n;s)$.
A {\em multifunctor} from a multicategory $\mathcal{C}$ with object set $S$ and $\mathcal{D}$
with object set $T$ consists of a map $F:S\r T$ and maps $\mathcal{C}(n)\r\mathcal{D}(n)$ over
products of copies of $F$, which preserves the composition. 

We will commonly use multicategories {\em enriched} in a symmetric monoidal category $Q$.
This is a variation of the notion of a multicategory where the set of objects $S$ remains a set,
but the multimorphisms $(s_1,\dots,s_n)\r s$ form objects $\mathcal{C}(s_1,\dots,s_n,s)$
of the category $Q$. Multicomposition then takes the form
$$\mathcal{C}(\dots)\otimes \mathcal{C}(\dots)\otimes\dots\mathcal{C}(\dots)\r \mathcal{C}(\dots)$$
where $\otimes$ is the symmetric monoidal structure in $Q$. The unit is a morphism
$$1\r \mathcal{C}(s,s)$$
where $1$ is the unit of $\otimes$. The required diagrams are still the obvious modifications
of diagrams expressing operad axioms. 

In particular, an ordinary multicategory is a multicategory enriched in the category of finite sets,
with the symmetric monoidal structure given by the Cartesian product. Enrichments over the categories
of topological spaces or simplicial sets are so common they are almost not worth mentioning. 
Multicategories enriched in the category of categories (or groupoids) and functors (with the Cartesian
product as symmetric monoidal structure) also occur in 
Elmendorf-Mandell's approach to multiplicative infinite loop space theory. These notions however
admit {\em weak versions}, which are trickier. A weak multicategory (or multifunctor) is a modification of the
respective notion  enriched in categories (or groupoids) where the respective axiom diagrams are
required only to commute up to natural isomorphisms (called {\em coherence isomorphisms}).
These isomorphisms must satisfy certain {\em coherence diagrams}, the precise definition
of which is technical, and will be relegated to the Appendix (Section \ref{sapp}), along with
a construction called {\em May-Thomason rectification} \cite{mt}, which allows us to replace them by
the corresponding strict notions in an appropriate sense. In Section \ref{twist}, we will also
briefly need to discuss weak versions of multicategories and multifunctors enriched in (strict)
2-categories. To complement weak multicategories, we shall call multicategories strictly enriched
in groupoids {\em strict}. Strict multicategories can be converted to multicategories
enriched in topological spaces by taking the nerve (bar construction) on 2-morphisms.
We shall denote this operation by $B_2$.

\vspace{3mm}
It may be good, at this point, to note that weak multifunctors form a weak $2$-category:
A {\em 1-morphism} of two weak multifunctors $\Phi$, $\Psi$ consists
of the following data: for objects $x$, a $1$-morphism $F:\Phi(x)\r\Psi(x)$
and for a $1$-multimorphism $M:(x_1,\dots x_n)\r y$, $2$-isomorphisms
$$\phi:M\circ (F,\dots, F)\cong F\circ M.$$
There are ``prism-shaped'' coherence diagrams required to be formed by these $2$-isomorphisms $\phi$ and
the coherence diagrams of multifunctors. 

A {\em $2$-isomorphism} of $1$-morphisms $F$, $G$ of multifunctors consists of
the following data: for every object $x$, a $2$-isomorphism $F(x)\cong G(x)$ which 
commute with the coherence isomorphisms $\phi$ of the $1$-morphisms $F$, $G$.

Recall further that a {\em weak isomorphism} (or {\em equivalence}) 
between two objects $x,y$ of a weak $2$-category is 
a pair of $1$-morphisms $x\r y$, $y\r x$ whose compositions are $2$-isomorphic to the identities.

\vspace{3mm}
We will also use the notion of {\em $\star$-categories}. A $\star$-category is a multicategory
in which for every $s_1,\dots,s_n\in S$, there exists a {\em universal} multimorphism
$\iota:(s_1,\dots, s_n)\r s_1\star\dots\star s_n$, i.e. for every multimorphism 
$\phi:(s_1,\dots \ s_n)\r t$, there exists a unique morphism $\psi:(s_1\star\dots\star s_n)\r t$
such that $\psi\circ \iota=\phi$ (here we write $\circ$ for the composition in the obvious sense.
The case of $n=0$ is included, we denoted the empty $\star$-product by $\mathbf{1}$.
A {\em $\star$-functor} is a weak multifunctor which preserves the $\star$-product.
(Note that since the $\star$-product is defined by universality, there is no need to discuss
coherences here.)

There is also a corresponding weak version in which $\psi\circ \iota\cong \phi$, and after a choice
of that natural 2-isomorphism, $\iota$ is determined up to unique 2-isomorphism making the resulting
2-diagram commute. The definition of a $\star$-functor remains unchanged.

A symmetric monoidal category with symmetric monoidal structure $\otimes$ determines a
$\star$-category by letting the multimorphisms $(a_1,\dots,a_n)\r a$ be the morphisms
$a_1\otimes\dots\otimes a_n\r a$. The notion of a $\star$-category, however, is more 
general. Let $a_{11},\dots a_{nk_n}$ be objects of a $\star$-category.
Then we have multimorphisms $(a_{i1},\dots,a_{ik_i})\r a_{i1}\star\dots\star a_{ik_i}$.
By the composition property, we then have a multimorphism
$$(a_{11},\dots ,a_{nk_n})\r (a_{11}\star\dots\star a_{1k_1})\star\dots\star (a_{n1}\star\dots\star a_{nk_n}).$$
By universality, we get morphisms
\beg{estar}{a_{11}\star\dots \star a_{nk_n}\r (a_{11}\star
\dots\star a_{1k_1})\star\dots\star (a_{n1}\star\dots\star a_{nk_n})}
and the category is symmetric monoidal when the morphisms \rref{estar} are all isomorphisms
(including the case when some of the $k_i$'s are equal to $0$).

By a {\em weak symmetric monoidal category} we shall mean a weak $\star$-category
in which the $1$-morphisms \rref{estar} are {\em equivalences}, which means that there
exists an inverse $1$-morphism with both compositions $2$-isomorphic to the identity.

\vspace{3mm}
The main purpose of using multicategories in our context comes from the work of Elmendorf and
Mandell \cite{em} who constructed a strict multicategory $Perm$ of permutative
categories, where multimorphisms are, roughly, multilinear morphisms (permutative categories are a version of symmetric monoidal categories where
the operation is strictly associative, see \cite{may}; the Joyal-Street construction \cite{js}, which
we will briefly discuss below in Section \ref{sapp}, allows
us to rectify symmetric monoidal categories into permutative categories).

Elmendorf and Mandell further discuss a {\em realization multifunctor} 
\beg{emultif}{\mathcal{K}:B_2(Perm)\r \mathcal{S}
}
where $\mathcal{S}$ is the topological symmetric monoidal category of symmetric spectra
(\cite{em}). The multifunctor $\mathcal{K}$ is unfortunately not a $\star$-functor. 
In fact, more precisely, the category $Perm$ is a $\star$-category, but only in the weak
sense. In effect, the weak $\star$-product of $n$ permutative categories $C_1,\dots,C_n$
has as objects formal sums
$$(a_{11}\otimes\dots \otimes a_{1n})\oplus\dots\oplus (a_{k1}\otimes\dots\otimes a_{kn})$$
where $a_{ij}\neq 0$, and is freely generated by ``tensor products'' of morphisms in the categories $C_i$
and the required coherences, modulo the coherence diagrams prescribed for multimorphisms in $Perm$.
On the other hand, a strong $\star$-product would require that all objects be of the form
$a_1\otimes\dots a_n$ (because those are the only objects on which the value of the
universal 1-morphism is prescribed exactly); this is clearly impossible except in special cases.

However,
the category $\mathcal{S}$ has a Quillen model structure \cite{ds}, and in particular 
a notion of equivalence; the functor $\mathcal{K}$ is a {\em homotopy $\star$-functor}
in the sense that the map
$$\mathcal{K}(a_1)\star\dots\star\mathcal{K}( a_n)\r \mathcal{K}(a_1\star\dots\star a_n)$$
coming from the multimorphism 
$$(\mathcal{K}(a_1),\dots, \mathcal{K}(a_n))\r \mathcal{K}(a_1\star\dots\star a_n)$$
by the fact that $\mathcal{K}$ is a multifunctor is an equivalence.

\vspace{3mm}

\section{Examples of multicategories. Naive modular functors}\label{smf}

In this section, we shall discuss a number of examples of weak $\star$-categories, and 
will define a modular functor as a weak $\star$-functor between appropriate weak
$\star$-categories (more precisely, a stack version will be needed to express 
holomorphic dependence, but we will discuss this when we get there).

The first kind of weak $\star$-categories which we will use are $1+1$-dimensional cobordism categories.
Recall that if a boundary component $c$ of a Riemann surface $X$ is parametrized by a
diffeomorphism $f:S^1\r c$ then $c$ is called {\em outbound} (resp. {\em inbound})
depending on whether the tangent vector $\iota$ at $1$ in the direction of $i$ is $i$ or
$-i$ times a tangent vector of $X$ on the boundary pointing outside.

The weak $\star$-category $\mathcal{A}^{top}$ has objects finite sets, multimorphisms from $S_1,\;\dots,\;S_n$
to $T$ Riemann surfaces with real-analytically parametrized inbound boundary components labeled by
$S_1\amalg\dots\amalg S_n$ and outbound boundary components labeled by $T$, 
and $2$-morphisms are isotopy classes of diffeomorphisms preserving orientation and
boundary parametrization. The operation $\star$ is, in fact, $\amalg$, and this makes
$\mathcal{A}^{top}$ a weakly symmetric monoidal category.

A variant is the weak $\star$-category $\mathcal{A}$ which has the same with the same
objects and multimorphisms as $\mathcal{A}^{top}$, but with $2$-morphism holomorphic
isomorphism preserving boundary component parametrization. In order for this weak multicategory
to be a weak $\star$-category, we must consider a disjoint union of finitely many copies of $S^1$
a (degenerate) Riemann manifold where the copies of $S^1$ are considered both inbound and outbound
boundary components (parametrized identically). Again, $\mathcal{A}$  is a weak symmetric monoidal category.

\vspace{3mm}
To model 
our concept which approaches most the original outline of Segal's concept of a modular \cite{segal},
consider the weak $\star$-category $\mathcal{C}$ whose objects are finite sets, multimorphisms
$(S_1,\dots, S_n)\r T$ are $T\times (S_1\times\dots\times S_n)$-matrices of finite-dimensional
$\C$-vector spaces, and $2$-morphisms are matrices of isomorphisms of $\C$-vector spaces.

Then, a {\em naive topological modular functor} is a weak $\star$-functor 
$$\mathcal{A}^{top}\r\mathcal{C}.$$
To eliminate the word ``topological'' means to replace $\mathcal{A}^{top}$ with $\mathcal{A}$,
but then we want to include some discussion of holomorphic dependence on the Riemann surface.
Then, there is a slight problem with the degenerate Riemann surfaces. We can,
for example, consider the category of subsets of $\C^n$ (with $n$ varying) which
are of the form $U\cup S$ where $U\subseteq \C^n$ is an open subset, and $S$ is a finite subset of
the boundary of $U$ and continuous maps 
$$U\cup S\r V\cup T$$ 
which are holomorphic on $U$. Consider the Grothendieck topology $\mathcal{G}$
on this category where
covers are open covers. Then consider the stack $\widetilde{\mathcal{A}}$ where sections over $U\cup {S}$ are 
maps continuous maps $f$ from $U\cup S$  into the Teichm\"{u}ller space of Riemann surfaces with
parametrized boundary components (including the degenerate Riemann surfaces)
which are holomorphic on $U$, and morphisms are continuous families of isomorphisms parametrized
over $U\cup S$ which are holomorphic on $U$. Consider also the stack $\widetilde{\mathcal{C}}$
whose sections over $U\cup {S}$ are continuous vector bundles over $U\cup S$ with holomorphic
structure on $U$, and continuous isomorphisms of vector bundles on $U\cup S$, which are
holomorphic on $U$.

Again provisionally, then, a {\em naive modular functor} is a morphism of stacks
$$\widetilde{\Phi}:\widetilde{\mathcal{A}}\r\widetilde{\mathcal{C}}$$
the sections of which over any $U\cup S\in Obj(\mathcal{G})$ are weak multifunctors.
We will continue to denote the sections over a point by
$$\Phi:\mathcal{A}\r\mathcal{C}.$$
The set $S=\Phi(*)$ is called the {\em set of labels}.
One also usually includes the normalization condition that there be a special label $1\in S$
where for the unit disk $D$ (with constant boundary parametrization), one $\Phi(D)(s)$
is $1$-dimensional for $s=1$ and trivial for $1\neq s\in S$.

\vspace{3mm}
It was  P.Deligne \cite{deligne1,deligne} who first discovered that this definition of a modular 
functor is insufficiently general in the sense that it does not include the case of the
Quillen determinant \cite{q}, which was meant to be one of the main examples discussed in
\cite{segal}. For the Quillen determinant, the set of labels has a single element $1$ and
the value of the multifunctor on any $1$-morphism is a $1$-dimensional $\C$-vector
space (called the {\em Quillen determinant line}), but Deligne observed
that in order for the gluing to work (in our language, for the multifunctor axioms to be satisfied),
the Quillen determinant line must be a {\em super-line}, i.e. must be given a $\Z/2$-grading
where a permutation isomorphism switching the factors in the tensor product of two odd lines is $-1$.
He further discovered that if this is allowed, the Quillen determinant line becomes non-canonical
(\cite{spin}). We introduce the machinery necessary to capture that situation in the next
section, but it turns out that the appropriate generality is even greater.

\vspace{3mm}

\section{Clifford algebras and modular functors}\label{real}

Deligne also noticed that the situation is even worse with the invertible chiral fermion (of central charge
$c=1/2$ - see Section\ref{twist} for a more detailed discussion of the central charge)
on Riemann surfaces with spin structure (\cite{spin}). Although this modular functor
is invertible under the tensor product, there is no consistent description in terms of
lines or super-lines, and one must consider irreducible Clifford modules. This leads
to the definitions we make in this section. With the most general definition, we will then
construct the K-theory realization.

\vspace{3mm}
\noindent
{\bf Remark:} A somewhat confusing
aspect of the chiral fermion is that there also exists a naive chiral fermion modular functor (of central
charge $c=1/2$) which has three labels \cite{fermion} and therefore is not invertible.
That example is of lesser significance to us, and will not be discussed further.

\vspace{3mm}
Let us recall that a {\em spin structure} on a Riemann surface $X$ (with boundary) is 
a square root of the tangent bundle $\tau$ of $X$, i.e. a complex holomorphic line bundle
$L$ together with an isomorphism $L_X\otimes_\C L_X \cong \tau_X$. A {\em spin-structure} on a
real $1$-manifold $Y$  is a real line bundle $L_Y$ together with an isomorphism
$L_Y\otimes_\R L_Y\cong \tau_Y$. It is important that a Riemann surface $X$ with spin structure
and with boundary canonically induces a spin structure on the boundary $\partial X$:
Let $L_{\partial X}$ consist of those vectors of $L_X|_{\partial X}$ whose
square is $i$ times a tangent vector of $X$ perpendicular to the boundary and pointing outside.

Recall that $S^1$ has two spin structures called {\em periodic} and {\em anti-periodic},
depending on whether the bundle $L_{S^1}$ is trivial or a M\"{o}bius strip.
The induced spin structure on the boundary of a disk is antiperiodic.
When parametrizing a boundary component $c$ of a Riemann surface with spin structure (we will,
again, restrict attention to real-analytic parametrizations), it is appropriate for our
purposes to specify a {\em parametrization with spin}, i.e. a diffeomorphism
$f:S^1\r c$ where $S^1$ is given a spin structure, together with an isomorphism
$L_{S^1}\r L_{c}$ over $f$, which squares to $Df$.

Now there are multicategories $\mathcal{A}_{spin}$ and $\mathcal{A}_{spin}^{top}$
whose objects are sets $S$ with a map to $\{A,P\}$, standing for ``periodic'' and ``antiperiodic'' 
(the inverse images of $A$, $P$ will be denoted by
$S_A$, $S_P$). $1$-multimorphisms $(S_1,\dots,S_n)\r T$ are Riemann surfaces with spin
with parametrized boundary components with
inbound resp. outbound boundary components indexed by $S_1\amalg\dots\amalg S_n$ 
resp. $T$, with matching spin structures. $2$-isomorphisms in $\mathcal{A}_{spin}$ are holomorphic isomorphisms $f$
with spin (i.e. with given square roots of $Df$) which is compatible with boundary parametrizations.
$2$-isomorphisms in $\mathcal{A}_{spin}^{top}$ are isotopy classes of diffeomorphisms with spin: 
for this purpose, it is more helpful to interpret spin structure equivalently as an 
$\widetilde{SL_2(\R)}$-structure on the tangent bundle where $\widetilde{SL_2(\R)}$ is
the double cover of $SL_2(\R)$. Then a diffeomorphism with spin is defined as a diffeomorphism
over which we are given a map of the associated principal $\widetilde{SL_2(\R)}$-bundles.

Again, in the case of $\mathcal{A}_{spin}$, 
$S^{1}_{A}$  and $S^{1}_{P}$, which are copies of $S^1$ with either spin structure,
must be considered to be (degenerate) Riemann surfaces with spin structure, where both
inbound and outound boundary components are the same, with identical parametrizations
(including identical spin). Then we can form a stack $\widetilde{\mathcal{A}_{spin}}$
analogously to the stack $\widetilde{\mathcal{A}}$ in the previous section, over the
same Grothendieck topology.

\vspace{3mm}
Based on ideas of P.Deligne \cite{deligne, spin}, to capture examples such as the invertible chiral
fermion, we introduce the weak $\star$-category $\mathcal{D}$ whose objects are data of the form $(S,A_s)$
where for every $s\in S$, $A_s$ is a super-central simple algebra, 
$1$-multi-morphisms $((S_1,A_s),\dots (S_n,A_s))\r (T,B_t)$ are 
$T\times (S_1\times\dots \times S_n)$-matrices, the $(t,(s_1,\dots, s_n))$ entry
being a 
$(B_t,A_{s_1}\otimes\dots\otimes A_{s_n})$-bimodule, and $2$-isomorphisms are
graded isomorphisms of bimodules.

Here, super-central simple algebras over $\C$ can be defined intrinsically, but for our purposes
we may define them as {\em Clifford algebras}, i.e. $\Z/2-graded$ algebras graded-isomorphic
to $\C$-algebras of the form $C_n=\C[x_1,\dots,x_n]/(x_{j}^{2}-1,x_jx_k+x_kx_j ,j\neq k)$, where the
degrees of the generators $x_j$ are odd. A Clifford algebra
is {\em even} (resp. {\em odd}) depending on whether $n$ is even or odd.
It is important to recall the {\em graded tensor product} of algebras or (bi)modules.
This is the ordinary tensor product, but the interchange map $T:M\otimes N\r N\otimes M$
is defined by
\beg{ekoszul}{T(x\otimes y)=(-1)^{deg(x)deg(y)}y\otimes x
}
on homogeneous elements $x\in M$, $y\in N$. This definition is applied in defining the $\Sigma_n$-action
on $1$-multi-morphisms, and also when making, for an $A$-module $M$ and $B$-module $N$,
$M\otimes N$ an $A\otimes B$-module. The same applies to bimodules, and is used in
defining the composition of $1$-multimorphisms in $\mathcal{D}$.

Similarly as in the last section, we have a stack $\widetilde{D}$ over the Grothendieck topology $\mathcal{G}$
where sections over $U\cup S$ are matrices of holomorphic bundles of 
$(B_t,A_{s_1}\otimes\dots\otimes A_{s_n})$-bimodules (for some
$(S_1,A_s),\dots (S_n,A_s), (T,B_t)$), which are interpreted as holomorphic principal bundles
with structure group 
$$\prod GL_{m(s_1,\dots,s_n,t)}(B_t\otimes (A_{s_1}\otimes\dots\otimes A_{s_n})^{Op})$$
where $m(s_1,\dots,s_n,t)$ are fixed non-negative integers.

We now define a {\em topological modular functor} as a $\star$-functor 
$$\mathcal{A}_{spin}^{top}\r\mathcal{D}.$$
Analogously to the last section, a {\em modular functor} is a morphism of stacks
$$\widetilde{\Phi}:\widetilde{\mathcal{A}_{spin}}\r\widetilde{\mathcal{D}}$$
the section of which over an object of $\mathcal{G}$ form a $\star$-functor.
The sections over a point will still be denoted by
$$\Phi:\mathcal{A}_{spin}\r \mathcal{D}.$$
The data $(S_A,A_s)=\Phi(*\mapsto A)$, $(S_P,A_s)=\Phi(*\mapsto P)$ are again 
referred to as sets of antiperiodic and periodic labels (decorated with Clifford algebras).
Again, one may include a normalization condition that there exists exactly one label
$(1,\C)$ on which the module $\Phi(D)$ where $D$ is the outbound unit disk $1$-multimorphism
is $1$-dimensional even, while on the other labels it is $0$.

\vspace{3mm}
\noindent
{\bf Example:} The {\em chiral fermion} of central charge $c=1/2$ which was considered
in \cite{spin} is an example of a modular functor in the sense just defined. This is
a theorem of \cite{spin}. In fact, this modular functor is {\em invertible} in the
following sense:

There is an operation of a tensor product of modular functors;
we call a modular functor $\Phi$ {\em invertible} if there exists a modular functor $\Psi$
such that $\Phi\otimes \Psi\sim 1$ where $1$ is the modular functor with one antiperiodic
and one periodic label, and all applicable $1$-multimorphisms going to $\C$. Here $\sim$
denotes a weak isomorphism of multifunctors covered by an equivalence of stacks. 

Invertible modular functors $(\widetilde{\Phi},\Phi)$ can be characterized as those
for which $\Phi$ factors as
$$\Phi_0:\mathcal{A}_{spin}\r \mathcal{D}_0\subset \mathcal{D},$$
where $\mathcal{D}_{0}$ is the sub- weak multicategory of $\mathcal{D}$ whose
objects are of the form $*\mapsto A$ for a Clifford algebra $A$, $1$-multimorphisms
are ($1\times $-matrices of) Morita equivalences, and $2$-isomorphism
are isomorphisms of bimodules. Recall that a {\em Morita equivalence} is a
graded $A,B$-bimodule $M$ such that $M\otimes_B ?$ (we use the symbol $?$ to denote an
unnamed variable) is an equivalence of
categories between finitely generated $B$-modules and finitely generated $A$-modules.

The weak multi-category $\mathcal{D}_0$ will play a role in the next section, in connection
with twistings of K-theory. In fact, invertible modular functors were completely classified
in \cite{spin}. In particular, it was proved there that all invertible modular functors are
weakly isomorphic to tensor products of tensor powers of the chiral fermion, and topological
modular functors. 

\vspace{3mm}
{\bf Remark:} At least conjecturally, there should be a large number of examples
which use the full generality of modular functors as defined here. For example,
supersymmetric modular functors, such as $N=1$ and $N=2$-supersymmetric
minimal models \cite{supermini}, which are irreducible representations of
of certain super-algebras containing the Virasoro algebra, are almost certainly
modular functors in our sense (and because of the super-symmetry, require the full scope of our formalism).

In this paper, we will discuss the case of the $N=1$ supersymmetric
minimal models, and will show that they give rise to at least a projective Clifford modular functors.

We do not discuss supersymmetry in this paper. One reason is that it requires some
work on super-moduli spaces of super-Riemann surfaces, which still has not been done
rigorously; the best reference available is the outline due to Crane and Rabin \cite{cr}.

\vspace{3mm}

\section{The Atiyah-Singer category and K-theory realization} \label{singer}

In this section, we will describe how one can extract K-theory information out of a modular
functor. The strategy is to construct a suitable weak $\star$-functor
\beg{ekk}{\mathcal{D}\r Perm,}
which, composed with the Elmendorf-Mandell multifunctor \rref{emultif}, would produce
a homotopy $\star$-functor
$$B_2\mathcal{D}\r\mathcal{S},$$
which could be composed with a topological modular functor to produce a functor
\beg{ekk1}{B_2\mathcal{A}_{spin}^{top}\r \mathcal{S}.}
(In case of modular functors, the superscript $top$ would be dropped.)

There is, in fact, an obvious construction in the case of naive modular functors: We may
define a weak $\star$-functor
\beg{ekk2}{\mathcal{C}\r Perm
}
on objects by 
$$S\mapsto \prod_S \C_2$$
where $\C_2$ is the category of finite-dimensional $\C$-vector spaces and isomorphisms
(topologized by the analytic topology on morphisms).
A $1$-multimorphism is then sent to the functor given by ``matrix multiplication'' (with respect
to the operations $\oplus$ and $\otimes$) by the given matrix of finite-dimensional vector
spaces. It is clear how matrices of $2$-isomorphisms correspond to natural isomorphisms
of functors. 

As mentioned above, composing with \rref{emultif}, we get a multifunctor
\beg{erealcs}{B_2\mathcal{C}\r \mathcal{S},}
but using the following trick of Elmendorf and Mandell, we can in fact improve this, getting a functor
into {\em $K$-modules} where $K$ denotes the $E_\infty$ ring spectrum of periodic K-theory 
(in fact, in the present setting, if we dropped the topology on
morphisms of $\C_2$,  it could just as well be the $E_\infty$ ring spectrum of
algebraic K-theory of $\C$, which enjoys an $E_\infty$ map into $K$).

Let $Q$ be any multicategory.
Consider a multicategory $\overline{Q}$ which has the objects of $Q$
and one additional object $*$. There is one multimorphism 
\beg{estars}{(\underbrace{*,\dots,*}_n)\r *}
for each $n$ and for every multimorphism
$$(a_1,\dots,a_n)\r b$$
in $Q$, a single multimorphism
\beg{estars1}{ (*,\dots,*, a_1,*,\dots,*,\dots a_n,*,\dots,*)\r b}
for every fixed numbers of $*$'s inserted between the $a_i$'s. Composition is obvious.
Similar constructions obviously also apply to enriched and weak multicategories.

Now if we have a weak multifunctor $\overline{Q}\r Perm$, then its restriction to 
the category with a single objects $*$ and multimorphisms \rref{estars}
realizes to an $E_\infty$ ring spectrum $R$, and the restriction weak multifunctor
$Q\r Perm$, which realizes to a multifunctor
$$B_2Q\r \mathcal{S},$$
is promoted to a multifunctor
$$B_2Q\r R-modules$$
(using the ``strictification'' Theorem 1.4 of \cite{em}).

In the case of \rref{ekk2}, we may define a weak multifunctor
\beg{ekk2a}{\overline{\mathcal{C}}\r Perm
}
simply by sending $*$ to $\C_2$. The 1-morphisms \rref{estars}, \rref{estars1} 
are sent simply to tensors with the vector spaces corresponding to the $*$ copies.
In this case, $R$ is connective K-theory $k$. Thus,  we can promote \rref{erealcs}
to a homotopy $\star$-functor to the multicategory of $k$-modules. By localizing with respect to the Bott
element (\cite{ekmm}), we can further pass from $k$-modules to $K$-modules.

\vspace{3mm}

This suggests to construct the weak multifunctor \rref{ekk} directly analogously
to \rref{ekk2}, i.e. to let 
$$ (S,A_s)\mapsto \prod_S A_s-Mod$$
where $A_s-Mod$ is the permutative category of finitely generated graded $A_s$-modules and graded
isomorphisms of modules. Indeed, this does produce a weak multifunctor of the form
\rref{ekk}, but {\em this is the wrong construction}; it does not,
for example, restrict to \rref{ekk2} (note that $\mathcal{C}$ is a sub- weak multicategory
of $\mathcal{D}$). In fact, permutative category of finitely generated graded $\C$-modules
is equivalent to the product of {\em two copies} of $\C_2$.

There is a good heuristic argument why no {\em finite-dimensional} construction of a multifunctor
\rref{ekk} can possibly be what we want. It relates to the weakly symmetric monoidal category
$\mathcal{D}_0$ which was discussed in the last section: The $E_\infty$ symmetric
monoidal category $B_2\mathcal{D}_0$ has two objects $\{even,odd\}$ (with the expected
product), and with automorphism groups of homotopy type $\Z/2\times K(\Z,2)$; it can
be interpreted as $B_2$ of the category of super-lines and isomorphisms (with the analytic
topology). This
suggests that the $E_\infty$ space $B(B_2\mathcal{D}_0)$ is a geometric model of the space of {\em twistings}
of $K$-theory as considered in \cite{fht}. We will, in fact, be able to make that more precise
below. 

For now, however, let us look at the chiral fermion modular functor example. If we want to construct
a realization of this into a weak $\star$-functor into $K$-modules, then the periodic label,
which the modular functor sends to an odd Clifford algebra, should be twisted by a shift of
dimension by $1$, i.e. it should be weakly equivalent to the $K$-module $\Sigma K$.
Therefore, the homotopy $\star$-functor 
$$B_2\mathcal{D}_0\r K-Mod$$
uses in substantial ways the relation
$$\Sigma K\wedge_K \Sigma K\sim K,$$
which is Bott $2$-periodicity. This clearly indicates that the construction cannot have a direct
algebraic K-theory analog, in which Bott periodicity is only valid with a Tate twist.

\vspace{3mm}
Therefore, one must bring to bear the full machinery of topological K-theory. The most convenient
model for this purpose seems to be a minor modification of the construction of Atiyah and Singer \cite{as}.
Let $C$ be a Clifford algebra (over $\C$). By a {\em Hilbert $C$-module} we shall mean a
$\Z/2$-graded complex (separable) Hilbert space $H=H_{even}\oplus H_{odd}$ together with
a morphism of graded $\Z/2$-graded $C^*$-algebras $C\r B(H)$ where $B(H)$ is the $\Z/2$-graded
$C^*$-algebra of bounded linear operators on $H$. (Recall that the canonical involution on
$C$ sends $x\mapsto x$ for $x$ even and $x\mapsto -x$ for $x$ odd.)

Now we shall define a symmetric monoidal category $\mathcal{F}(C)$ in which both
the sets of objects and morphisms are topologized; some basic facts about such categories,
including the Joyal-Street construction (making them into permutative categories)
will be discussed in the Appendix (Section \ref{sapp}).

The space $Obj(\mathcal{F}(C))$ is a disjoint union over (finite or infinite-dimensional)
Hilbert $C$-modules $H$ of spaces $\mathcal{F}(H)$ defined as follows:
When $C$ is even, $\mathcal{F}(H)$ consists of all homogeneous odd skew self-adjoint Fredholm operators
$F:H\r H$ which anticommute with all odd elements of $C$. When $C$ is odd, 
$\mathcal{F}(H)$ consists of all homogeneous odd skew self-adjoint Fredholm operators
$F:H\r H$ which anticommute with all odd elements of $C$ such that $iF$ is neither positive definite
nor negative definite on any subspace of finite codimension. (Note that in the odd case,
this in particular excludes the possibility of $H$ being finite-dimensional.) In both cases, the topology on
$\mathcal{F}(H)$ is the induced topology from $B(H)\times K(H)$ via the map 
$F\mapsto (F, 1+F^2)$ where $B(H)$ is given the weak topology and $K(H)$ is the
space of compact operators on $H$ with the norm topology. 
(At this point, we could equivalently just use the norm topology, but the more refined topology
described above, which is due to Atiyah-Segal \cite{aseg} is needed when considering 
the stack version of the multifunctor \rref{ekk} which we are about to define.)

The space $Mor(\mathcal{F}(C))$ is 
a disjoint union over pairs $(H,K)$ of Hilbert $C$-modules of the spaces
$$\mathcal{F}(H)\times Iso(H,K)$$
where $Iso(H,K)$ is the space of metric isomorphisms of Hilbert $C$-modules
with the norm topology. (Recall that when $H,K$ are infinite-dimensional, then $Iso(H,K)$ is
contractible by Kuiper's theorem.) 

Now the category $\mathcal{F}(C)$ is symmetric monoidal with the operation of direct sum $\oplus$.
By a theorem of Atiayh and Singer \cite{as}, in fact, the spectrum associated with the symmetric
monoidal category $\mathcal{F}(C)$ is $k$ when $C$ is even and $\Sigma k$ when $C$ is odd.
(Thus, localizing at the Bott element produces the spectra we need.)

The weak $\star$-functor \rref{ekk} can now be constructed as follows: On objects,
we put
$$(S,A_s)\mapsto \prod_{s\in S} \mathcal{F}(A_s).$$
On 1-morphisms, we let a multimorphism $((S_1, A_s),\dots,(S_n,A_s))\r (T,B_t)$
given by a matrix of bimodules $M_{t,(s_1,\dots,s_n)}$ send an $n$-tuple
$$((H_s, F_s|s\in S_1),\dots,(H_s,F_s|s\in S_n))$$
of pairs consisting of a Hilbert space and skew self-adjoint Fredholm operator
to the $n$-tuple indexed by $t$ of Hilbert $A_t$-modules
$$\bigoplus_{s_i\in S_i} M_{t,(s_1,\dots, s_n)}\otimes_{A_{s_1}\otimes\dots \otimes 
A_{s_n}}H_{s_1}\widehat{\otimes}\dots\widehat{\otimes} H_{s_n}$$
(where $\widehat{\otimes}$ is the Hilbert tensor product) with 
the corresponding ``product of skew self-adjoint Fredholm operators'', as defined in \cite{as}.
This completes the definition of \rref{ekk}, and hence, in particular, of \rref{ekk1} for
an arbitrary topological modular functor. 

We may again use the Elmendorf-Mandell trick
to extend the weak multifunctor \rref{ekk} to $\overline{\mathcal{D}}$ by sending
$*$ to $C_2$; combining with localization at the Bott element, we promote \rref{ekk1} to
a weak $\star$-functor
$$B_{2}\mathcal{A}_{spin}^{top}\r K-Mod,$$
which, for topological modular functors, is what we were asking for. For a general modular functor,
we get, of course, the same thing with $\mathcal{A}_{spin}^{top}$ replaced by
$\mathcal{A}_{spin}$, but that is somewhat unsatisfactory. We may, in fact, 
consider a Hilbert bundle version (following the lines of \cite{fht}) of the construction to obtain a sheaf version,
but on the $K$-module side, we do not know how to preserve the holomorphic information, so
the sections of 1-morphisms over a space $Y$ will be modules over $Map(Y,K)$. In addition
to losing holomorphic information,
this will only be a presheaf of spectra, satisfying the sheaf condition up to homotopy. 
Therefore, we clearly want to say something better for modular functors which are not topological.
We will address that in the next section.

\vspace{3mm}
To conclude the present section, let us note that by restricting 
the functor \rref{ekk} to $\overline{\mathcal{D}}$, 
and then restricting to $\overline{\mathcal{D}_0}$, we obtain
a homotopy $\star$-functor 
$$B_2(\mathcal{D}_0)\r K-modules.$$
This is one model of the ``action'' of the $E_\infty$ space $B(B_2\mathcal{D}_0)$ on
the category of $K$-modules, as mentioned above.

\vspace{3mm}

\section{Topological twisting and remarks on classification}\label{twist}

In this section, we will address the question how to extract topological 
information from a (not necessarily topological) modular functor. Let $\widetilde{\mathcal{D}}_{proj}$
be the {\em projective version} of the stack $\widetilde{\mathcal{D}}$, i.e.
sections over $U\cup S$ are principal bundles (holomorphic over $U$)
with structure group equal to the product
over the individual matrix entries
of the (topological) automorphism groups of the respective Clifford modules, factored out
by $\C^\times$, acting (on all the matrix entries simultaneously) by scalar multiplication.
We see that the section of $\widetilde{\mathcal{D}}_{proj}$ over an object of $\mathcal{G}$
form a multicategory.

Note also that $\mathcal{A}_{spin}^{top}$ can also be promoted to a stack
$\widetilde{\mathcal{A}_{spin}^{top}}$ which is simply equivalent to the quotient stack
$[*/\Gamma]$ where $\Gamma$ is the appropriate mapping class group.
We have then canonical projections of stacks
$$p:\widetilde{\mathcal{D}}\r \widetilde{\mathcal{D}}_{proj}$$
and
$$q:\widetilde{\mathcal{A}_{spin}}\r\widetilde{\mathcal{A}_{spin}^{top}}.$$
The key observation is the following result due to G. Segal \cite{segal}:

\begin{lemma}\label{l1}
Consider a modular functor $\Phi$ (with the normalization condition), 
there is a canonical morphism of stacks completing the following diagram:
\beg{estacks}{
\diagram
\widetilde{\mathcal{A}_{spin}}\rto^{\widetilde{\Phi}} \dto_q &\widetilde{\mathcal{D}}\dto^p\\
\widetilde{\mathcal{A}_{spin}^{top}}\rdotted|>\tip_{\widetilde{\Phi_{proj}}}& \widetilde{\mathcal{D}}_{proj}
\enddiagram
}

\end{lemma}

\vspace{3mm}

\noindent
{\em Proof sketch:} 
The idea is to cut out a small disk from a Riemann surface with boundary,
and glue in an annulus (which can vary along a parametric set $U\cup S$).
Then the modular functor takes on a non-zero value only for the unit label on the
cut, and its value on the annulus (which must have the same labels on both
boundary components) is also $1$-dimensional. The gluing
isomorphism then establishes a projective trivialization
of the matrix of bundles given by $\widetilde{\Phi}$ on the given section of 
$\widetilde{A}_{spin}$ over $U\cup S$. 

To prove consistency, we must show that the projective trivialization constructed
does not depend on the choice of the holomorphic disk we cut out. To this end, consider
a pair of pants with unit label on all boundary components. By the gluing isomorphism,
again, the value of the modular functor on the pair of pants is $1$-dimensional. 
This shows that the projective trivializations obtaining by gluing annuli on either
of the boundary components of the pair of pants coincide.

\qed

\vspace{3mm}
In diagram \rref{estacks}, since the source of the bottom row is topological, we can drop
the $\widetilde{?}$, i.e. it suffices to consider the sections over a point
$$\Phi_{proj}:\mathcal{A}_{spin}^{top}\r\mathcal{D}_{proj}.$$
But how to realize this data topologically? While we could factor out $\C^\times$ from the morphisms
of the categories $\mathcal{F}(C)$, (which is clearly related to twisted K-theory),
those categories are no longer symmetric monoidal, so it is not clear how to apply
the infinite loop space machine of Elmendorf and Mandell.

To remedy this situation, we shall, instead of factoring out the $\C^\times$ from the 
morphisms, add it as 3-morphisms to the target multicategory. More specifically, 
we consider a weak multicategory $\overline{\mathcal{D}}$
strictly enriched in groupoids, by which we mean a structure satisfying 
the axioms of a weak multicategory where the $2$-morphisms between
two $1$-multimorphisms form a groupoid, and composition is functorial. 
The definition of $\overline{\mathcal{D}}$ is the same as the definition
of $\mathcal{D}$ with an added space of $3$-morphisms which is the
Cartesian product of the space of $2$-morphisms and $\C^\times$: 
An element $\lambda\in \C^\times$ acts on a $2$-morphism by scalar multiplication.

\vspace{3mm}
We also introduce the concept of a {\em 2-weak multifunctor} 
$$Q\r W$$
where $Q$ is a weak multicategory and $W$ is a weak multicategory strictly enriched
in groupoids in the above sense. This the weak version of the concept of a weak
multifunctor, considered as a weak morphism of multisorted algebras of operad type
with objects and $1$-morphisms fixed,
as considered below in Section \ref{sapp}. In other words, the $2$-morphisms
satisfy the axioms of a weak multifunctor where every equality of 2-morphisms
prescribed by that structure is replaced by a 3-isomorphism. $3$-isomorphisms are
then required to satisfy coherence diagrams corresponding to situations where
one operation on $2$-morphisms in the concept of a weak multifunctor can be
converted to another by a sequence of relations required by the structure
in two different ways.

From this point of view, Lemma \ref{l1} gives a 2-weak multifunctor
\beg{e2weak}{\overline{\Phi}:\mathcal{A}_{spin}^{top}\r \overline{\mathcal{D}}.
}
Note that the data \rref{e2weak} are completely topological! We will refer to 
a 2-weak multifunctor \rref{e2weak} as a {\em projective modular functor}.

To construct a topological realization of \rref{e2weak}, let $C$ be a Clifford algebra.
We construct a category strictly enriched in groupoids $\overline{\mathcal{F}(C)}$
to have the same objects and 1-morphisms as $\mathcal{F}(C)$, and we let the space of $2$-morphisms
be the Cartesian product of the space of $1$-morphisms with $\C^\times$; the 2-morphisms
act, again, by scalar multiplication.

Then $\overline{\mathcal{F}(C)}$ is not a symmetric monoidal category enriched
in groupoids: there is no way of adding two different $2$-morphisms. Consider instead
the strict $2$-category $O^2\C^\times$ whose spaces of objects and $1$-morphisms are $*$,
and the space of 2-morphisms is $\C^\times$. Then there is an obvious ``forgetful'' strict
2-functor 
\beg{aweakU}{U: \overline{\mathcal{F}(C)}\r O^2\C^\times,}
which gives $\overline{\mathcal{C}}$ the structure of a symmetric monoidal category strictly
enriched over groupoids over $O^2\C^\times$. Recall that a symmetric monoidal category $H$
over a category $K$ is a functor $H\r K$, a unit $K\r H$ and a product 
$\oplus: H\times_K H\r H$ which satisfy the usual axioms of a symmetric monoidal
category. The version strictly enriched in groupoids is completely analogous.
Further, the (topological) Joyal-Street construction allows us to rectify each
symmetric monoidal category over a category into a permutative category over a category
(which means that the symmetric monoidal structure is strictly associative unital), and
similarly for the version strictly enriched in groupoids. Analogously to the theorem
of Elmendorf and Mandell \cite{em}, we then have a strict multicategory 
$$Perm/Cat$$
of permutative categories over categories, and the corresponding version strictly enriched in
groupoids
$$(2-Perm)/(2-Cat).$$

Now we may construct from \rref{e2weak} a weak multifunctor strictly enriched
in groupoids
\beg{e3weak}{\overline{\mathcal{D}}\r (2-Perm)/(2-Cat)}
simply by the same construction as we used for
\rref{ekk}, where on the level of $3$-morphisms, we define composition by multiplication 
in $\C^\times$. Using $\overline{\Phi}$, we then obtain a 2-weak multifunctor
\beg{e4weak}{\mathcal{A}_{spin}^{top}\r (2-Perm)/(2-Cat).
}
Using the techniques described in Section \ref{sapp},
this can be rectified into a weak multifunctor strictly enriched in groupoids, so 
rectifying the $2$-level and applying $B_3$,
we get a weak multifunctor
\beg{e5weak}{\mathcal{A}_{spin}^{top}\r Perm/Cat.
}
Using the machine of Section \ref{sapp} again, we can convert this to a strict multifunctor.

Now there is a relative version of the Elmendorf-Mandell machine, which produces a multifunctor
\beg{emrel}{B_2 Perm/Cat\r \text{Parametrized symmetric spectra.}
}
The construction is on the formal level a fairly straightforward analog of the construction
of Elmendorf-Mandell \cite{em}, although setting up a full model structure on symmetric
parametrized spectra is actually quite tricky (see \cite{lind,msig}). We omit the details, as
this would make the present paper disproportionately long. 

Using also an analogue of the Elmendorf-Mandell module trick in the category of parametrized spectra,
sending the label $*$ to the 2-permutative category $\overline{\C_2}$ over $O^2\C^\times$ given by
finite-dimensional $\C$-vector spaces, isomorphisms and isomorphisms $\times \C^\times$,
we can then obtain a multifunctor
\beg{e6weak}{B_2\mathcal{A}_{spin}^{top}\r \text{modules over twisted K-theory}
}
where by twisted K-theory we mean the parametrized $E_\infty$ ring spectrum $K/K(\Z,3)$,
constructed by applying the relative Elmendorf-Mandell machine to $\overline{\C_2}$,
and localizing fiber-wise at the Bott element.

Note that because of the construction we used, we can actually say more about what happens
to the twisting in \rref{e6weak}. Let us compose \rref{e4weak} with the forgetful functor 
$$(2-Perm)/(2-Cat)\r(2-Cat),$$
obtaining a 2-weak multifunctor
\beg{e7weak}{\mathcal{A}^{top}_{spin}\r (2-Cat).
}
Denote by $\mathcal{O}^2\C^\times $ the strong multicategory enriched in groupoids
which has only one object and morphisms $O^2\C^\times$. Then we have an obvious 
forgetful strict multifunctor enriched in groupoids
\beg{eweakuu}{\overline{\mathcal{D}}\r \mathcal{O}^2\C^\times.
}
Consider the 2-weak multifunctor
\beg{e8weak}{\Phi_{twist}:\mathcal{A}^{top}_{spin}\r \mathcal{O}^2\C^\times
}
given by the composition of \rref{e2weak} with \rref{eweakuu}. We will call the 2-weak multifunctor
\rref{e8weak} the {\em topological twisting} associated with $\widetilde{\Phi}$.
By definition, \rref{e7weak} is determined by \rref{e8weak}: the objects
go to $O^2\C^\times$, and the morphism go to the product, multiplied by another copy
of $O^2\C^\times$ determined by \rref{e8weak} on morphisms. Thus, we obtain the following

\begin{theorem}
\label{t1}
A modular functor $\widetilde{\Phi}$ determines a multifunctor
\beg{e10weak}{|\widetilde{\Phi}|:B_2\mathcal{A}_{spin}^{top}\r K/K(\Z,3)-modules.}
Furthermore, the topological twisting determines a multifunctor $\phi$ from $B_2\mathcal{A}_{spin}^{top}$
to the multicategory with objects $*$ and morphisms $K(\Z,3)$ (considered as an
abelian group), with composition given by
$K(\Z,3)$-multiplication; on underlying spaces $K(\Z,3)$, \rref{e10weak} on a space of multimorphisms
is given by the product in the abelian group $K(\Z,3)$, multiplied additionally by $\phi$.
\end{theorem}

\qed

\noindent
{\bf Explanation:} Recall \cite{msig} that for a map of spaces $f:X\r Y$, there is a pullback functor
parametrized spectra 
$$f^*:Spectra/Y\r Spectra/X$$
which has a left adjoint denoted by $f_\sharp$ and a right adjoint denoted by $f_*$.
(The situation with parametric modules is the same.) The map on multimorphisms given
by Theorem \ref{t1} can be described as follows. Denote by $\mu=\mu_n:K(\Z,3)^{\times
n}\r K(\Z,3)$ the multiplication, and let $\wedge$ denote the external smash-product
of parametrized spectra (i.e. sending a parametrized spectrum over $X$ and a parametrized
spectrum over $Y$ to a parametrized spectrum over $X\times Y$). We have the map
$$\phi:B_2\mathcal{A}^{top}_{spin}(S_1,\dots,S_n;T)\r K(\Z,3).$$
Let 
$$\pi:B_2\mathcal{A}^{top}_{spin}(S_1,\dots,S_n;T)\times K(\Z,3)\r K(\Z,3)$$
be the projection, and let 
$$j:B_2\mathcal{A}^{top}_{spin}(S_1,\dots,S_n;T)\times K(\Z,3)\r K(\Z,3)$$
be given by
$$(x,y)\mapsto \phi(x)\cdot y.$$
(Note that $B_2\mathcal{A}^{top}_{spin}(S_1,\dots,S_n;T)$ is homotopically equivalent
to the classifying space of a mapping class group.) Then the map on multimorphisms
given by Theorem \ref{t1} is a map of $K/K(\Z,3)$-modules of the form
$$\begin{array}{l}j_\sharp\pi^*(\mu_\sharp(|\widetilde{\Phi}|(S_1)\wedge\dots\wedge|\widetilde{\Phi}|(S_n))
\wedge_{\mu_\sharp (K/K(\Z,3)\wedge...\wedge K/K(\Z,3))}K/K(\Z,3))\\
\r |\widetilde{\Phi}|(T).
\end{array}$$

\vspace{3mm}

\noindent
{\bf Example:} At present, projective modular functors are easier to construct than
modular functors. For example, modular tensor categories, as defined
in \cite{bk}, give rise to projective versions
of the ``naive'' modular functors considered in Section \ref{smf}.
The authors of \cite{bk} give an exact statement of coherence diagrams of a modular
functor as an exercise to the reader, and Section \ref{smf} of the present paper
can be interpreted as one approach to a solution of that exercise. The proof 
of \cite{bk}, giving a passage from a modular tensor category to a projective version
of a modular functor in the sense of Section \ref{smf}, in any case, applies.

From this point of view, we can make contact with the work of Freed-Hopkins-Teleman \cite{fht}.
They compute the equivariant twisted $K$-theory groups $K^{*}_{G,\tau}(G)$
where $\tau$ is a ``regular'' twisting in an appropriate sense, and $G$ is a compact Lie
group acting on itself by conjugation. At least for, say, compact Lie groups with
torsion free fundamental group, this coincides with the Verlinde algebra obtained by
taking dimensions of the vector spaces in the (naive - although there is also an $N=1$-super-symmetric
version) modular functor corresponding
to the chiral WZW model.

The chiral WZW model is known to give rise a modular tensor category (\cite{bk, huang05}),
and hence gives rise to a projective modular functor. The Verlinde algebra in this case
has been computed in the physics literature (see \cite{fermion} for a survey and
original references), and the known answer has been proved by \cite{fht} to coincide
with $K^{*}_{G,\tau}(G)$. While we do not know if there is a reference of the WZW
Verlinde algebra computation which conforms fully with mathematical standards of rigor,
from a foundational point of view, the existence of a modular tensor category is the
deeper question; the treatment of the fusion rules computation in the physics literature
using singular vectors in the discrete series Verma modules
over Kac-Moody algebra is, in our opinion, essentially correct.

By those computations, then, we {\em know} that
the twisted K-theory realization of the projective modular functor associated with
the chiral WZW models is
the parametrized spectrum $(K_{G,\tau}(G))^G$ over $K(\Z,3)$, and
furthermore we know that the up to homotopy, the composition product and unit given
by Theorem \ref{t1} coincides with the product constructed in \cite{fht}.

Still, it would be nice to have an even more direct geometrical connection. 
For example, the homotopical interpretation of the Verlinde algebra product and unit \cite{fht}
was also shown by the Kriz and Westerland \cite{kwest} to relate
to the product and unit of a twisted K-theory version of Chas-Sullivan's string topology \cite{chs}.
Surprisingly, however, it was shown in \cite{kwest} that the coproduct in twisted K-theory string
topology does {\em not} give the right answer for a coproduct coming from a modular
functor, and there is no augmentation in string topology at all. Therefore, perhaps the
first operation in $(K_{G,\tau}(G))^G$ one should try to find a purely topological description of 
is the augmentation. As far as we know, no such description is known.

\vspace{3mm}

To conclude this section, we say a few words about the classification of the possible
topological twistings \rref{e8weak}. To talk about classification, we must introduce a notion
of equivalence of topological twistings. This, however, is implicit in what
we already said: 
A topological twisting \rref{e8weak} is a $2$-weak multifunctor, which
can be considered a weak multisorted algebra of operadic type where the
variables are images of $2$-morphisms of $\mathcal{A}_{spin}^{top}$. 
We therefore have a notion of a weak morphism of such structures
(see \ref{ssrect}). We call two topological twistings {\em equivalent}
if a weak morphism exists between them. (Since the target is $\mathbb{C}^\times$,
a weak morphism in the opposite direction automatically exists, too.)
Note that the data specifying an equivalence of topological twistings
will then be a map $h:\mathcal{A}^{top}_{spin}\r O^2\C^\times$.

Recall (a fact from \cite{segal} which is readily reproduced in our present formalism) 
that a modular functor determines a holomorphic $\C^\times$-central extension of the
semigroup of annuli. Recall further from \cite{segal} that those $\C^\times$-central extensions
are classified by a single complex number called the {\em central charge}. (This is the same
number classifying the $\C$-central extensions of the Witt algebra of polynomial complex vector
fields on $S^1$, although a formal passage between both contexts requires some technical
care due to the fact that we are dealing with infinitely many dimensions.) In fact, one must
prove that the central charge does not depend on label, but this can be done by taking
an annulus with a given label, and cutting out a disk (which must have unit label). Comparing
the variations of the different boundary components of the resulting pair of pants shows that
any label has the same central charge as the unit label.

We then have the
following

\begin{proposition}
\label{p1}
The topological twisting of a modular functor is, up to equivalence, completely determined by its central charge. 
The central charges of invertible modular functors are integral multiples of $1/2$. The central
charges of modular functors which have trivial topological twisting (up to equivalence) are
precisely the multiples of $4$.
\end{proposition}  

\begin{proof}
The first statement follows completely from the proof of Lemma \ref{l1}, since the modular
functor on the annuli we glue in is a line bundle $L$  determined, by definition, by the central
charge. Over an object of $\mathcal{G}$, the bundle given by the modular functor
is trivialized after tensoring with $L$. Therefore, modular functors with the same central
charge produce equivalent data.

The second statement was proved in \cite{spin}.

For the third statement, note from \cite{spin} that central charge $4$ is realized by (the inverse of)
the square of the Quillen determinant. This modular functor takes values only in (even) lines.
Therefore, the topological twisting data are trivial, since the modular functor itself is
a multifunctor $\mathcal{A}^{top}_{spin}\r O^2\C^\times$, whose topological
twisting is by definition $0$.

On the other hand, if the topological twisting of a modular functor is trivial up to equivalence,
then, by definition, the modular functor becomes topological (i.e. acquires central charge $0$)
after being tensored with an invertible modular functor which takes values only in even lines.
It follows from the classification in \cite{spin} that such invertible modular functors are
isomorphic to even powers of the Quillen determinant.

\end{proof}

\vspace{3mm}
\noindent
{\bf Comment:} Proposition \ref{p1} is interesting in part because we have produced a topological
invariant which characterizes the central charge {\em modulo a certain number}. New speculations
\cite{bdh} about possible use of powers of the fermion conformal field theory for a geometric construction
elliptic cohomology predict such a phenomenon (although we do not expect here to recover the
exact periodicity of topological modular forms, in part due to the fact we omitted real structure and
other refinements). It should be pointed out, however, that the formalism of modular
tensor categories treated for example in \cite{bk} also gives an exponential of
an imaginary multiple the central charge as a ``topological invariant'', namely the data
contained in the modular tensor categories. We do not know whether modular tensor
categories can be extended into a formalism which would fully determine the stack data
of a holomorphic modular functor, and the central charge. This is why in the examples constructed
in the next two sections, we will restrict attention to projective (Clifford) modular functors.

\vspace{3mm}

\section{Clifford modular tensor categories and projective modular functors}
\label{sexamp}

In this Section, we shall discuss a method for constructing modular functors in the generality
involving Clifford algebras. For simplicity, we will restrict attention to projective modular functors,
which are sufficient for constructing the twisted K-theory realizations discussed in the
last Section. We only know how to do this in a somewhat
roundabout way. The fact is that currently, direct constructions of modular functors out
of analytical data assigned to a Riemann surface, such as in the case of the chiral fermion
\cite{spin}, are generally unknown. The best known results on rigorous constructions
of projective modular functors come from modular tensor categories, using the vertex
operator algebra, and Huang's theorem \cite{huang05}.

In this section, we will discuss how to apply these methods in the Clifford case. Perhaps
surprisingly, we will not define a concept of a ``modular tensor category with spin''. 
One reason is that to use such a notion,
one would have to develop a separate discussion of Moore-Seiberg type
constraints \cite{bk}, Chapter 5.2 for Riemann surfaces with Spin structure. Another
reason is that in the Clifford modular case, in some cases, the s-matrix corresponding
to an elliptic curve with Kervaire invariant 1 (i.e. on which every non-separating simple closed
curve is periodic) is singular (in the case of the chiral fermion, this $1\times 1$ matrix is $0$).
We will explain what causes this ``paradox'', and how to get around it, a little later on.

It turns out that instead, projective Clifford modular functors come from ordinary modular
tensor categories with certain extra structure. For the definition of a modular tensor
category, we refer the reader to \cite{bk}. 

\vspace{3mm}
\noindent
{\bf Definition:} A {\em pre-Clifford modular tensor category} is a modular tensor category $C$
with product $\boxtimes$ and an object $V^-$ together with an isomorphism
\beg{ecmodiso}{\diagram \iota:V^-\boxtimes V^-\rto^(.65)\cong & 1\enddiagram}
which satisfies
\beg{ecmodminus}{\theta_{V^-}=-1.
}
(For the definition of $\theta$, see \cite{bk}, Chapter 2. Note that \rref{ecmodiso} implies $(\theta_{V^-})^2=1.$)

\vspace{3mm}
In what follows, we will always assume that we are in a pre-Clifford modular tensor
category as described in the definition.

\vspace{3mm}
\begin{lemma}
\label{lcmtc1}
We have $(\sigma_{V^-V^-})^2=1$.
\end{lemma}

(For the definition of $\sigma$, see \cite{bk}, Section 1.2.)

\begin{proof}
Compute:
$$\begin{array}{l}
1=\theta_{V^-\boxtimes V^-}=\sigma_{V^-V^-}\sigma_{V^-V^-}(\theta_{V^-}\otimes \theta_{V^-})\\
=(\sigma_{V^-V^-})^2.
\end{array}$$
\end{proof}

\vspace{3mm}
\begin{lemma}
\label{lcmtc2}
Let $M\in Obj(C)$. Put
$$\zeta_M=\sigma_{V^-M}\sigma_{MV^-}.$$
Then $(\zeta_M)^2=1$.
\end{lemma}

\begin{proof}
By the braiding relation, we have
\beg{ecmtc1}{\begin{array}{l}
\sigma_{V^-(V^-\boxtimes M)}\sigma_{(V^-\boxtimes M)V^-}\\[2ex]
=(\sigma_{V^-V^-}\boxtimes Id_M)(Id_{V^-}\boxtimes \sigma_{V^-M}\sigma_{MV^-})
(\sigma_{V^-V^-}\boxtimes Id)\\[2ex]
=Id_{V^-}\boxtimes \sigma_{V^-M}\sigma_{MV^-}.
\end{array}
}
Now compute:
$$\begin{array}{l}
1\boxtimes 1\boxtimes \theta_M=
\theta_{V^-\boxtimes V^-\boxtimes M}\\[2ex]
= \sigma_{V^-(V^-\boxtimes M)}\sigma_{(V^-\boxtimes M)V^-}(\theta_{V^-}\boxtimes
\theta_{V^-\boxtimes M})\\[2ex]
=\sigma_{V^-(V^-\boxtimes M)}\sigma_{(V^-\boxtimes M)V^-}
(\theta_{V^-}\boxtimes (\sigma_{V^-M}\sigma_{MV^-}(\theta_{V^-}\boxtimes \theta_M)))\\[2ex]
=(1\boxtimes\zeta_M)(1\boxtimes \zeta_M\theta_M)\\[2ex]
=\theta_{V^-}^{2}\boxtimes \zeta_{M}^{2}(1\boxtimes\theta_M)=1\boxtimes 
\zeta_{M}^{2}(1\boxtimes\theta_M).
\end{array}
$$
Thus, $\zeta_{M}^{2}=Id_{V^-\boxtimes M}$, as claimed.
\end{proof}

\vspace{3mm}
If $M$ is irreducible, then $\zeta_M\in\mathbb{C}^\times$, so $\zeta_M\in \{\pm 1\}$.

\vspace{3mm}

\begin{lemma}
\label{lcmtc3}
Let $M,N\in Obj(C)$ be irreducible. Then
$$\zeta_{M\boxtimes N}=\zeta_M\cdot \zeta_N.$$
\end{lemma}

\begin{proof}
Using the braiding relation,
$$
\begin{array}{l}
\sigma_{V^-(M\boxtimes N)}\sigma_{(M\boxtimes N)V^-}\\[2ex]
=(\sigma_{V^-N}\boxtimes Id_N)(Id_M\boxtimes \sigma_{V^-N}\sigma_{NV^-})
(\sigma_{MV^-}\boxtimes Id_N)\\[2ex]
=\zeta_N(\sigma_{V^-M}\boxtimes Id_N)(\sigma_{MV^-}\boxtimes Id_N)=\zeta_N\zeta_M.
\end{array}
$$
\end{proof}

\vspace{3mm}
We call an irreducible object $M$ {\em Neveu-Schwarz (or NS)} (resp. {\em Ramond (or R)})
if $\zeta_M=1$ (resp. $\zeta_M=-1$).

\vspace{3mm}

\begin{lemma}
\label{lcmtc4}
An irreducible object is NS (resp. R) if and only if
$$\theta_{V^-\boxtimes M}=-Id_{V^-}\boxtimes\theta_M$$
resp.
$$\theta_{V^-\boxtimes M}=Id_{V^-}\boxtimes\theta_M.$$
\end{lemma}

\begin{proof}
We have
$$\theta_{V^-\boxtimes M}=\zeta_M(\theta_{V^-}\boxtimes \theta_M)=-\zeta_M(Id_{V^-}\boxtimes 
\theta_M).$$
\end{proof}

\vspace{3mm}
Next, we will make some observations on counting isomorphism classes of irreducible
R and NS objects (also called {\em labels}),
and the $s$-matrix. To this end, we need some additional notation. Note that $V^-\boxtimes ?$
defines an involution $\overline{?}$ of isomorphism classes of irreducible objects of $C$. The fixed points
of the involution are all R. We denote the set of fixed points by $R^0$, and call them
{\em non-split R labels}. The regular orbits can consist of NS or R labels. Choose a set
$NS^+$ of representatives of regular NS orbits, and a set of representatives $R^+$ of R orbits.
Let also $NS^-=\overline{ NS^+}$, $R^-=\overline{ R^+}$. The elements of $NS=NS^+\cup NS^-$
will be called {\em NS labels}, the elements of $R^\pm=R^+\cup R^-$ {\em split R labels}.

\vspace{3mm}

\begin{lemma}
\label{cmtc5}
(1) Let $i$ be a label. Then
\beg{ecmtci}{
s_{\overline{i}j}=\left\{\begin{array}{rl}
s_{ij} & \text{if $j\in NS$}\\
-s_{ij} & \text{if $j\in R$.}
\end{array}\right.
}
(2) If $i\in R_0, j\in R$, then
\beg{ecmtciv}{s_{ij}=0.}
\end{lemma}

\begin{proof}
We have
\beg{ecmtc+}{
s_{ij}=\theta_{i}^{-1}\theta_{j}^{-1}\sum_{k}N_{ij}^{k}\theta_kd_k
}
where $d_k$ is the quantum dimension (see \cite{bk}).
We have
$$N_{\overline{i}j}^{\overline{k}}=N_{ij}^{k},$$
and also
$$d_{\overline{k}}=d_k\cdot d_{V^-}=d_k,$$
since $V^-$ is invertible and hence its quantum dimension is $1$ (\cite{dong}). Also,
if $j\in NS$, then $i,k$ are both NS or both R by Lemma \ref{lcmtc3}. Thus,
$$\theta_i\theta_k=\theta_{\overline{i}}\theta_{\overline{k}}.$$
Similary, if $j\in R$, then by Lemma \ref{lcmtc3}, one of the labels $i,k$ is NS
and the other is R. Thus,
$$\theta_i\theta_k=-\theta_{\overline{i}}\theta_{\overline{k}}.$$
Consequently, \rref{ecmtci} follows from \rref{ecmtc+}. To prove \rref{ecmtciv},
just note that for $i\in R_0$, $j\in R$, by \rref{ecmtci}, we have $\overline{i}=i$, so
$$s_{ij}=s_{\overline{i}j}=-s_{ij}.$$
\end{proof}

\vspace{3mm}
By Lemma \ref{cmtc5}, classified by the type of labels, the s-matrix has the following form:
\beg{etable}{\parbox{2.3in}{
\begin{tabular}{l|c|c|r|r|r|}
\protect& $NS^+$&$NS^-$&$R^+$&$R^-$&$ R^0$\\
\hline
&&&&&\\
$NS^+ $&$A$&$A$&$B$&$B$&$D$\\
\hline
&&&&&\\
$NS^-$&$A$&$A$&$-B$&$-B$&$-D$\\
\hline
&&&&&\\
$R^+$&$B^T$&$-B^T$&$C$&$-C$&$0$\\
\hline
&&&&&\\
$R^-$ &$B^T$ &$-B^T$ &$-C$&$C$&$0$\\
\hline
&&&&&\\
$R^0$ & $D^T$ & $-D^T$ &$0$&$0$&$0$\\
\hline
\end{tabular}}
}

\vspace{3mm}
\begin{proposition}
\label{pcmtc1}
The matrices $A$ and $C$ of Table \ref{etable} are symmetric and non-singular.
The matrix $(\begin{array}{cc} B&D\end{array})$ is non-singular, and we have 
\beg{epcmtc1}{B^TD=0.}
We also have
\beg{epcmtc2}{|R^+| +|R^0|=|NS^+|.}
\end{proposition}

\begin{proof}
The s-matrix is symmetrical, hence so are the matrices $A$, $C$.
By performing row and column operations on the matrix \rref{etable}, we may obtain
the matrix
$$
\left(
\begin{array}{lcrrr}
0&0&0&B&D\\
0& 4A & 0&0&0\\
0& 0& C & 0 & 0\\
B^T &0&0&0&0\\
D^T &0&0&0&0
\end{array}
\right)
$$
Thus, the statements about non-singularity of matrices follow, since the s-matrix is
non-singular. Therefore, \rref{epcmtc2} follows. To prove \rref{epcmtc1}, recall
that the square of the s-matrix is a scalar multiple of the charge conjugation matrix,
and note that obviously, $R^0$ and $R^\pm$ are invariant under the operation
of taking contragredient labels.
\end{proof}

\vspace{3mm}
\begin{lemma}
\label{lcmf1}
The following diagram commutes:
$$
\diagram
V^-\boxtimes V^-\boxtimes V^-\dto^{\cong}_{1\boxtimes \iota}\rto^(.6){\iota\boxtimes 1}_(.6){\cong}
&
1\boxtimes V^-\dto^{\eta}_{\cong}\\
V^-\boxtimes 1\rto^{\cong}_{\eta} & V^-
\enddiagram
$$
where $\eta$ is the unit coherence isomorphism.
\end{lemma}

\begin{proof}
Recall that by Lemma \ref{lcmtc1} and irreducibility, $\sigma_{V^-V^-}=\pm 1$. Thus,
by the braiding relation,
$$\eta(\iota\boxtimes 1)=\eta(1\boxtimes\eta)(\sigma_{V^-V^-}\boxtimes 1)(1\boxtimes
\sigma_{V^-V^-})=\eta(1\boxtimes\iota).$$
\end{proof}

\vspace{3mm}

\begin{lemma}
\label{lcmf2}
Let $X\in R^0$. Then there exists an isomorphism 
$$\theta:V^-\boxtimes X\r X$$
so that the following diagram commutes:
$$
\diagram
V^-\boxtimes V^-\boxtimes X\dto_{\iota}\rto^(.6){1\boxtimes \theta}
&
 V^-\boxtimes X\dto^{-\theta}\\
1\boxtimes X\rto_{\eta} & X.
\enddiagram
$$
\end{lemma}

\begin{proof}
By irreducibility, the diagram commutes up to multiplication by a non-zero complex
number $\lambda$. Hence, it suffices to replace $\theta$ by $\theta/\sqrt{\lambda}$.
\end{proof}

\vspace{3mm}
From now on, we will assume that a choice of $\theta$ has been made as in Lemma \ref{lcmf2}.

\vspace{3mm}
\noindent
{\bf Definition:} A pre-Clifford modular tensor category is called a {\em Clifford tensor
category} if
\beg{ecmtccliff}{\sigma_{V^-V^-}=-1.
}

\vspace{3mm}
\noindent
{\bf Remarks:} 1. We do not know whether there exists pre-Clifford modular tensor
categories in which both $R^0\neq\emptyset$ and $R^\pm\neq \emptyset$.

2. In the remainder of this Section, we will produce a Clifford projective modular functor
from a Clifford tensor category. The assumption \rref{ecmtccliff} is
essential to our arguments.
The alternative to \rref{ecmtccliff} is, by Lemma \ref{lcmtc1}, 
\beg{ecmtcsqrt}{\sigma_{V^-V^-}=1.}
In the paper \cite{spin}, there naturally appeared, as an alternative to Spin structure
on Riemann surfaces something called the ``Sqrt structure'', which is, vaguely speaking,
an ``untwisted analogue'' of Spin structure. We believe that replacing \rref{ecmtccliff}
with \rref{ecmtcsqrt} could be used to produce a notion of a modular functor
on Riemann surfaces with Sqrt structure, but do not follow this direction in detail,
since it appears to be currently of lesser importance from the point of view of examples.

\vspace{3mm}
The remainder of this section from this point on is dedicated to constructing a projective Clifford
modular functor on the weak $\star$-category of Riemann surfaces with Spin structure
from a Clifford modular tensor category $C$. Of course $C$, being a modular tensor
category, by the construction
of \cite{bk}, Chapter 5, in particular defines an ordinary projective modular functor $M$ on
the weak $\star$-category of Riemann surfaces (without Spin structure). This phenomenon
is well known in mathematical physics. For example, in \cite{fqs}, the modular functor $M$
is referred to as the {\em spin model}. Later, it became more widely known as the
{\em GSO projection}.

We denote by $S(X_1,\dots,X_n)$ a standard sphere (\cite{bk}, Section 5.2) with $n$
punctures $\{1,\dots,n\}$ oriented outbound, labelled by $n$ irreducible objects 
$X_1,\dots, X_n\in Obj(C)$. Denote, for an irreducible object $X\in Obj(C)$,
$\epsilon\in \Z/2$,
$$X(\epsilon)=\left\{
\begin{array}{rl}
X &\text{if $\epsilon=0$}\\
V^-\boxtimes X & \text{if $\epsilon=1$.}
\end{array}\right.
$$
Put
$$\begin{array}{l}\widetilde{M}(S(X_1,\dots,X_n))\\[2ex]
=\displaystyle\bigoplus_{\begin{array}[t]{c}
\epsilon_i\in\Z/2\\i=1,\dots,n
\end{array}}M(X_1(\epsilon_1),\dots,X_n(\epsilon_n))\\[8ex]
=\displaystyle\bigoplus_{\begin{array}[t]{c}
\epsilon_i\in\Z/2\\i=1,\dots,n
\end{array}}Hom_C(1,X_1(\epsilon_1)\boxtimes\dots\boxtimes X_n(\epsilon_n))).
\end{array}
$$
We may consider $\widetilde{M}(X_1,\dots,X_n)$ as a $(\Z/2)^n$-graded vector space by 
$(\epsilon_1,\dots,\epsilon_n)$. We may also consider a ``total'' $\Z/2$-grading by 
$\epsilon_1+\dots+\epsilon_n$. Denote by
$1\leq i_1<\dots<i_k\leq n$ all those indices $i$ such
that $X_i$ We will construct commuting involutions
$\alpha_i$, $i=1,\dots,n-1$ of $(\Z/2)^n$-degree 
$$(\underbrace{0,0,\dots,0}_{i-1},1,1,0,\dots,0),
$$
and {\em anticommuting} involutions $\lambda_j$ of $(\Z/2)^n$-degree
$$(\underbrace{0,0,\dots,0}_{i_j-1},1,0,\dots,0),
$$
where $\alpha_i$, $\beta_j$ commute for any $i=1,\dots,n$, $j=1,\dots, k$.

\vspace{3mm}
The operator $\alpha_i$ is of the form
$$\begin{array}{l}
\bigoplus Hom_C(1,\\
Id_{X_1(\epsilon_1)}\boxtimes\dots\boxtimes
Id_{X_{i-1}(\epsilon_{i-1})}\boxtimes q
\boxtimes Id_{X_{i+2}(\epsilon_{i+2})}\boxtimes\dots\boxtimes Id_{X_{n}(\epsilon_{n})})
\end{array}$$
where 
$$q:X(\epsilon_1)\boxtimes Y(\epsilon_2)\r X(\epsilon_1+1)\boxtimes Y(\epsilon_2+1)$$
is given, according to the different values of $\epsilon_1,\epsilon_2\in \Z/2$, as follows:
$$
\diagram
X\boxtimes Y\rrto^q\drto_{\iota^{-1}\boxtimes 1} &&V^-\boxtimes X\boxtimes V^-\boxtimes Y\\
&V^-\boxtimes V^-\boxtimes X\boxtimes Y\urto_{1\boxtimes\sigma_{V^-X}\boxtimes 1}&
\enddiagram
$$

$$
\diagram
X\boxtimes V^-\boxtimes Y\dto_{\iota^{-1}\boxtimes 1\boxtimes 1}\rrto^q &&
V^-\boxtimes X\boxtimes Y\\
V^-\boxtimes V^-\boxtimes X\boxtimes V^-\boxtimes Y\rrto_{1\boxtimes \sigma_{V^-X}\boxtimes 1
\boxtimes 1}
&&
V^-\boxtimes X\boxtimes V^-\boxtimes V^-\boxtimes Y\uto_{1\boxtimes 1\boxtimes\iota \boxtimes 1}
\enddiagram
$$

$$
\diagram
V^-\boxtimes X\boxtimes Y\rrto^{q=\sigma_{V^-X}\boxtimes 1}&&V\boxtimes V^-\boxtimes Y
\enddiagram
$$

$$
\diagram
V^-\boxtimes X\boxtimes V^-\boxtimes Y\rrto^q\drto_{\sigma_{V^-X}\boxtimes 1\boxtimes1}
&&
X\boxtimes Y\\
& X\boxtimes V^-\boxtimes V^-\boxtimes Y.\urto_{1\boxtimes \iota\boxtimes 1}&
\enddiagram
$$

\noindent
The operator $\lambda_j$ is
$$
\bigoplus (-1)^{(\sum_{s=1}^{i_j-1}\epsilon_s)}
Hom_C(1,
\underbrace{1\boxtimes\dots\boxtimes 1}_{\text{$i_j-1$ times}}
\boxtimes \lambda\boxtimes 1\boxtimes\dots\boxtimes 1)
$$
where for $X\in R^0$, $\epsilon\in \Z/2$,
$$ 
\lambda: X(\epsilon)\r X(\epsilon+1)
$$
is given as follows:
$$
\diagram
V^-\boxtimes X\rrto^{\lambda=\theta}&&X
\enddiagram
$$

$$
\diagram
X\rrto^{\lambda=\theta^{-1}}
\dto_{\alpha\boxtimes 1} && V^-\boxtimes X\\
V^-\boxtimes V^-\boxtimes X\rrto_{\sigma_{V^-V^-}\boxtimes X}
&&
V^-\boxtimes V^-\boxtimes X\uto^{1\boxtimes \theta}
\enddiagram
$$
(this diagram commutes by Lemma \ref{lcmf2}.

\vspace{3mm}

\begin{lemma}
\label{lclifcom}
The operators $\iota_i$, $i=1\dots n$ commute and satisfy $(\iota_i)^2=1$. The
operators $\lambda_j$, $j=1,\dots, k$ anticommute and satisfy $(\lambda_j)^2=1$.
Moreover, every operator $\alpha_i$ commutes with every operator $\lambda_j$.
\end{lemma}

\begin{proof}
A straightforward computation using Lemmas \ref{lcmf1}, \ref{lcmf2} and the braiding
relation.
\end{proof}

\vspace{3mm}
Now note that each of the involutions $\alpha_i$, since it is $\Z/2$-graded of odd
degree with respect to grading by the $i$'th copy of $\Z/2$, is diagonalizable,
and half of its eigenvalues are $+1$, half are $-1$. Moreover, since $\alpha_i$
commute, they are simultaneously diagonalizable.

Moreover, each $\alpha_1\times\dots\times \alpha_{n-1}$-weight
$(w_1,\dots,w_{n-1})\in (\Z/2)^{n-1}$ corresponds to a spin structure 
on $S(X_1,\dots,X_n)$ as follows: On the boundary circle of the $i$'th
puncture, put an antiperiodic (resp. periodic) Spin structure depending
on whether $X_i$ is NS (or R). Furthermore, identify the spinors on all of the
points $P_1,\dots, P_n$ of each of the circle of lowest imaginary part with $\R$. 
On the path from $P_i$ to $P_{i+1}$ along the marking graph
of $S(X_1,\dots, X_n)$ (see \cite{bk}, Section 5.2), put the antiperiodic
resp. periodic Spin structure depending on whether $w_i=-1$ or $w_i=1$.
Denote the resulting spin structure on $S(X_1,\dots,X_n)$ by
$\sigma(w_1,\dots, w_n)$.

Let
$$M({S(X_1,\dots,X_n),\sigma(w_1,\dots,w_{n-1})})$$
be the $(w_1,\dots,w_{n-1})$ weight space of
$$\widetilde{M}({S(X_1,\dots,X_n),\sigma(w_1,\dots,w_{n-1})})$$
with respect to $(\alpha_1,\dots,\alpha_{n-1})$, considered as
a left module over
\beg{ecmflambda}{\Lambda= T_\C(\lambda_1,\dots,\lambda_j)/(\lambda_i\lambda_k=
-\lambda_k\lambda_i,\lambda_{i}^{2}=1)}
(where $T_\C$ denotes the $\C$-tensor algebra on the given generators).

In discussing the passage from a modular tensor category to a modular functor,
\cite{bk} do not discuss orientation of boundary components in detail. This
is because reversal of orientation of a boundary component can be
always accomplished by changing a label to its contragredient label.

In the Spin case, however, we need to be more careful because reversal
of orientation of a periodic boundary component does not carry 
a canonical Spin structure: A cylinder with two inbound (or two outbound)
boundary components has two different possible Spin structures which
are interchanged by a diffeomorphism interchanging te boundary
components. 

To discuss reversal of orientation, recall that in the standard sphere
$S(X_1,\dots,X_n)$ with a spin structure $\sigma=\sigma(w_1,\dots,w_{n-1})$,
the boundary components decorated by the labels $X_1,\dots, X_{n-1}$ 
were oriented outbound. We may create a {\em mirror} sphere 
$\overline{S}(X_{1}^{*},\dots,X_{n}^{*})$ 
(together with a canonical Spin structure $\overline{\sigma}$) 
by reflecting by the imaginary axis, and replacing labels
by contragredient ones. (Note: To facilitate gluing, we only need to consider
$n=2$.)

Let
$$M(\overline{S}(X_{1}^{*},\dots,X_{n}^{*}),\overline{\sigma})
=Hom_{\Lambda}(M(S(X_1,\dots,X_n),\sigma),\Lambda)$$
where $\Lambda$ is the Clifford algebra \rref{ecmflambda}. Thus,
$M(\overline{S}(X_{1}^{*},\dots,X_{n}^{*}),\overline{\sigma})$ is
naturally a {\em right} $\Lambda$-module, hence a left $\Lambda^{Op}$-module).
Note that $\Lambda^{Op}$ is isomorphic to $\Lambda$, but not canonically. In fact,
using the Koszul signs, we have, canonically,
$$(T_\C(\alpha)/\alpha^2=1)^{Op}=T_\C(\alpha^*)/((\alpha^*)^2=-1).$$
However, for every $\C$-algebra $A$, $A\otimes A^{Op}$ has a canonical bimodule
(naturally identified with $A$) from either side, and applying this to $\Lambda$ facilitates
gluing of an inbound and outbound boundary component with Spin structure.

Now the effect of gluing and moves on $M(\Sigma,\sigma)$ for labelled 
surfaces $\Sigma$ with spin structure $\sigma$ (up to scalar multiple)
follows from the corresponding statement on the GSO projection, taking
into account the change of Spin structure caused by the move.
(This is why we don't need a separate ``lego game'' for surfaces with Spin structure.)
All the statements are straightforward consequences of the definition, and we omit
the details.

\vspace{3mm}
One case, however, warrants special discussion, namely the S-move. We have proved
above in Proposition \ref{pcmtc1} that the s-matrices corresponding to elliptic
curves of Kervaire invariant $0$ (NS-NS and NS-R) are non-singular. In the case
of the elliptic curve of Kervaire invariant $1$ (R-R), we only know that the $R^+$-$R^+$
s-matrix is non-singular, while the rest of the s-matrix is zero!

To explain this effect, note that when gluing 
$$S(X_1,X_2),\sigma$$ 
to
$$\overline{S}(X_{1}^{*},X_{2}^{*}),\overline{\sigma}$$ 
where $X_1,X_2$ are non-split Ramond
(note: we necessarily have $X_1=X_{2}^{*}$), the curve spanned by the two marking
graphs of $S(X_1,X_2),\sigma$ and $\overline{S}(X_{1}^{*},X_{2}^{*}),\overline{\sigma}$ 
is antiperiodic, since we are gluing the boundary at the angle $\pi$ and not $0$.
The trace, in this case ,then, is a copy of $\C$ for each $R_0$ label: this is
the R-NS elliptic curve).

To obtain the R-R curve, we replace one of the Spin structures, say, $\overline{\sigma}$,
with the other possible Spin structure on $\overline{S}(X_{1}^{*},X_{2}^{*})$. This
results in the {\em reversal of signs} of the action of one of the generators $\Lambda_1$
or $\Lambda_2$ (depending on how exactly we identify the spinors at the points $P_1,P_2$, which
is also non-canonical).

In any case, one readily verifies that if we take the trace
after this modification of Spin structure, we get $0$. Thus, the vector space assigned by
our construction to the Kervaire invariant $1$ elliptic curve is, in fact, the free
$\C$-moduls on the set of $R^+$-labels!

\vspace{3mm}
\noindent
{\bf Remark:} There are examples of Clifford modular tensor categories
with split ($R^\pm$) labels (for example an even power of the chiral fermion), and examples
of Clifford modular tensor categories with non-split ($R^0$) labels (for example an odd
power of the chiral fermion). We do not know, however, an example of a Clifford
modular tensor category which would have both split and non-split R labels.

\vspace{5mm}

\section{Super vertex algebras, $N=1$ supersymmetric minimal models.}
\label{svoa}

In this section, we will describe how Clifford modular tensor categories (and hence
projective Clifford modular functors) may be obtained from super vertex 
algebras, and we will specifically discuss the example of $N=1$ supersymmetric
minimal models. For a definition of a super vertex algebra, and basic facts about
this concept, we refer the reader to Kac \cite{kac}. In this paper, we will
only consider (super) vertex algebras with a conformal element $L$ (also denoted
by $\omega$, cf. \cite{kac}). For a super vertex algebra $V$, we denote by $V^+$
(resp. $V^-$) the submodule of elements of weights in $\Z$ (resp. $(1/2) +\Z$).

\vspace{3mm}

\begin{theorem}
\label{tvoa}
Let $V$ be a super-vertex algebra with a conformal element $L$ which satisfies Huang's
conditions \cite{huang05}:

(1) $V_{<0}=0$, $V_{0}=\C$ and the contragredient module to $V$ is $V$.

(2) Every $V$-module is completely reducible.

(3) $V$ satisfies the $C_2$-condition. (Explicitly, the quotient of $V$ by the sum of
the images of $aV$ where $a$ is a coefficient of $z^{\geq 1}$ of the vertex operator $Y(u,z)$,
$u\in V$, is finite-dimensional.)

Then the category of finitely generated $V^+$-modules is a Clifford modular tensor category
(and consequently, by the construction of the last Section produces an example
of a projective Clifford modular
functor.)
\end{theorem}

\begin{proof}
Clearly, condition (1) passes on to $V^+$. Conditions (2) and (3) pass on to $V^+$ by the
results of Miyamoto \cite{mia1,mia2,mia3,mia4}, applying them to the case
of the $\Z/2$ acting on $V$ by $1$ on $V^+$ and $-1$ on $V^-$. (While Miyamoto does not discus
super vertex algebras explicitly, his arguments are unaffected by the generalization.)

Thus, it remains to prove that $V^-\boxtimes V^-\cong V^+$ in the category of $V^+$-modules.
The super vertex algebra structure gives a canonical map
\beg{esvoa1}{\diagram
V^-\boxtimes V^-\rto^\mu & V^+.
\enddiagram
}
The map must be onto since $V^+$ is an irreducible $V^+$-modules, and if the image
of \rref{esvoa1} were $0$, $V^-$ would be an ideal in $V$.

Suppose $\mu$ is not injective. Let $M=Ker(\mu)$. Note that any $V^+$-module $X$ with
a map $V^-\boxtimes X\r X$ which satisfies associativity with the $V^+$-module structure and the map
and the map \rref{esvoa1} is a weak $V$-module. (In this proof, $\boxtimes$ means the fusion
tensor product in the category of $V^+$-modules.) Thus, in particular,
$$V\boxtimes V^-$$
is a weak $V$-module, and the map 
$$\phi: V\boxtimes V^-\r V$$
given by right multiplication by the $V^+$-module $V^-$ is a map of weak $V^+$-modules. Consequently,
$M=Ker(\phi)$ is a weak $V^+$-module and hence, by dimensional considerations, a $V^+$-module.
Also for dimensional reasons, $V^-$ annihilates the $V$-module $M$. Hence, the $V$-annihilator
of $M$ is a non-trivial ideal in $V$, which is a contradiction. 

Thus, $M=0$ and $\mu$ is injective.
\end{proof}

\vspace{3mm}
\noindent
{\bf Example:} The $N=1$ supersymmetric minimal model is a super vertex algebra obtained
as a quotient $L_{p,q}$ of the Verma module $V(c_{p,q},0)$ of the $N=1$ Neveu Schwarz algebra $\mathcal{A}$
(for a definition, see e.g. \cite{gk})
by the maximal ideal, where
\beg{ecpq}{c_{p,q}=\frac{3}{2}\left(1-\frac{2(p-q)^2}{pq}
\right),}
$p,q\in \Z_{\geq2}$, $p\equiv q\mod 2$ and $gcd(p,(p-q)/2)=1$. Non-isomorphic irreducible NS (resp. R)
modules are given by the $N=1$ minimal models $L_{p,q,r,s}$ with central charge $c$, i.e.
quotients of the Verma module $V(c_{p,q},h_{r,s})$ over the NS (resp. R) algebra (for a definition
of the R-algebra, see e.g. \cite{ik}) where $1\leq r\leq p-1$, $1\leq r\leq q-1$, $r,s\in\Z$ and
$r\equiv s\mod 2$ (resp. $r+1\equiv s\mod 2$) and
$$h_{r,s}=\frac{(pr-qs)^2-(p-q)^2}{8pq}+\frac{\epsilon}{16}$$
where $\epsilon=0$ (resp. $\epsilon=1$). Additionally, we have
$$L_{p,q,r,s}\cong L_{p,q,p-r,q-s}$$
(which, for dimensional reasons, are the only possible isomoprhisms between these irreducible
modules). Thus, when both $p,q$ are odd, there are
$$\frac{(p-1)(q-1)}{4}\;\;\text{NS (resp. R) irreducible modules,}$$
and when $p,q$ are both even (in which case $p-q \equiv 2\mod 4$), there are
$$\frac{(p-1)(q-1)+1}{4}\;\;\text{NS (resp. R) irreducible modules}.$$
By a result of Zhu \cite{zhu}, a vertex algebra $V$ cannot have more irreducible modules
than the dimension of $V/C_2V$. Furthermore, when equality arises,
the Zhu algebra is a product of copies of $\C$, and hence is semisimple.
This is the case of $V=L_{p,q}$ by Theorem \ref{tc2}. Hence, we also know that 
the above list of irreducible modules $L_{p,q,r,s}$ is complete.

To prove the condition of complete reducibility (condition (2) of Theorem \ref{tvoa}),
it then suffices to show that
$$Ext^{1}(M,N)=0$$
for any two irreducible modules $M$ and $N$. This is proved in \cite{gk} for the case of NS modules.
Since there can only be non-trivial extensions if $M,N$ are both NS or both R, assume
that $M,N$ are both R (the argument we are about to give works in both cases).
Let $\mathcal{A}^{R}_{-}$ be the subalgebra of the Ramond algebra $\mathcal{A}^R$
spanned by $G_{\geq 0}$, $L_{\geq 0}$. Then we have a BGG resolution of $M$
by Verma modules
$$V_{h}=\mathcal{A}^{R}\otimes_{\mathcal{A}^{R}_{-}} V_{h}^{0}$$
where on $V_{h}^{0}$, $\mathcal{A}^{R}_{-}$ acts through its $0$ degree, and
$V^{0}_{h}$ is $2$-dimensional, with $L_0$ acting by $h$, and $G_0$ acting by
$\pm\sqrt{h}$ on the two basis elements. 

Then the BGG resolution of $M$ has the form
$$\dots\r\bigoplus_k V_{h_{i_2},k}\r\bigoplus_k V_{h_{i_1},k}\r V_h,$$
$$h_{i_j,k}>h,\; h_{i_j,k}\in h+\Z.$$
This leads to a spectral sequence
\beg{ebgg+}{E_{1}^{p,q}=\prod_k Ext^{q}_{\mathcal{A}^{R}_{-}}
(V_{h_{i_p},k},N)\Rightarrow Ext^{p+q}_{\mathcal{A}^R}(M,N).
}
For dimensional (integrality) reasons, \rref{ebgg+} can only be non-zero for $M=N$
and the only terms we need to worry about are $p=0,q=1$ and $p=1,q=0$.
The former is excluded by the fact that $h$ is the lowest weight of $M=N$, 
the latter by the fact that $N$ has no singular vectors. 

Thus, the asumptions are verified, and the $N=1$ supersymmetric minimal models
give examples to which Theorem \ref{tvoa} applies. Note that by Lemma \ref{lsingv},
(the parity of the number of $G$'s), it follows that all the R labels are split when
$(p-1)(q-1)$ is even, and non-split when $(p-1)(q-1)$ is odd.

\vspace{5mm}

\section{Appendix: Some technical results}\label{sapp}

We shall describe here some constructions needed to make rigorous the
results of the previous sections. We will start with May-Thomason rectification \cite{mt}.

\vspace{3mm}

\subsection{Rectification of weak multicategories}

The idea is to approach the problem much more generally. A {\em universal algebra of operadic
type $\mathcal{T}$} is allowed to have any set $I$ (possibly infinite) of $n_i$-ary operations
$\circ_i$ ($n_i\geq 0$ finite), $i\in I$ and any set $J$ (possibly infinite) of {\em relations} of
the form
\beg{eoper1}{w_j(x_1,\dots,x_{m_j})=w_{j}^{\prime}(x_1,\dots,x_{m_j}),\; j\in J
}
where $w_j$, $w_{j}^{\prime}$ are finite ``words'' one can write using the different variables
$x_1,\dots, x_{m_j}$ and the operations $\circ_i$, such that on both sides of \rref{eoper1}, every
variable $x_1,\dots,x_{m_j}$ occurs precisely once (some operations on the other hand
may be repeated, or may not occur at all). 

A $\mathcal{T}$-algebra then is a model of this universal algebra structure, 
i.e. a set with actually $n_i$-ary operations $\circ_i$, which satisfy the relations \rref{eoper1},
when we plug in concrete (possibly repeating) elements for $x_1,\dots,x_{m_j}$. 

For a universal algebra of operadic type $\mathcal{T}$, there exists a canonical operad
$\mathcal{C}_\mathcal{T}$ such that the category of $\mathcal{T}$-algebras is canonically
equivalent to the category of $\mathcal{C}_\mathcal{T}$-algebras. In fact, if we introduce
the smallest equivalence relation on words which is stable with respect to substitutions,
and such that the left and right hand sides of \rref{eoper1} are equivalent words, then
we have
\beg{eoper1a}{\mathcal{C}_{\mathcal{T}}(n)=\{ \text{equivalence classes of words in 
$x_1,\dots, x_n$}\}
}
We shall call the type $\mathcal{T}$ {\em free} if the operad $\Sigma_n$-action on
$\mathcal{C}_\mathcal{T}$ is free. 

\vspace{3mm}

If $\mathcal{T}$ is a universal algebra of operadic type, then a {\em weak $\mathcal{T}$-algebra}
is a groupoid with functorial operations $\circ_i$, $i\in I$ where each equality 
\rref{eoper1} is replaced by a natural isomorphism (called {\em coherence isomorphism}).
These coherence isomorphisms are required to form {\em coherence diagrams} which are
described as follows: Suppose we have a sequence $w_0,\dots, w_m$, 
$w_m=w_0$ of words in non-repeating variables $x_1,\dots, x_n$, each of which
is used exactly once, such that
\beg{eoper2}{\begin{array}{l}
w_{k}(x_1,\dots,x_n)= w(\dots,w_{j_k}(\dots)\dots)\\
w_{k}(x_1,\dots,x_n)= w(\dots,w_{j_k}^\prime(\dots)\dots),
\end{array}
}
$k=0,\dots, m-1$
(i.e. at each step, we make a change along \rref{eoper1}, combined with substitutions
(which means we can substitute into the variables inside the word, or the whole word
may be used as a variable for further operations, as long as, again, every variable
$x_1,\dots, x_n$ ends up used exactly once in the whole word). Then there is
an obvious coherence diagram modelled on the sequence $w_1,\dots,w_m$.

\vspace{3mm}
If $\mathcal{T}$ is a free operadic type of universal algebras, then a weak $\mathcal{T}$-algebra
can be rectified into a $\mathcal{T}$-algebra by the following construction due to May
and Thomason \cite{mt}: First, consider the (strict) operad enriched in groupoids $\mathcal{C}^\prime$
which is the free operad $\mathcal{O}_\mathcal{C}$ on the sequence of sets $(\mathcal{C}(n))_{n\geq 0}$,
where we put precisely one isomorphism 
$$x\cong x^\prime$$
for $x,x^\prime\in \mathcal{O}_{\mathcal{C}}(n)$ such that
$$\epsilon(x)=\epsilon(x^\prime)\in \mathcal{C}(n)$$
where $\epsilon:\mathcal{O}_{\mathcal{C}}\r\mathcal{C}$ is the counit of the adjunction between
the forgetful functor from operads to sequences and the free operad functor.

Now denoting by $|\mathcal{C}^\prime|(n)$ the nerve (bar construction) on $\mathcal{C}^\prime(n)$,
$\epsilon$ induces a map of operads
\beg{eoper3}{\iota:|\mathcal{C}^\prime|\r\mathcal{C}
}
which is an equivalence on each $n$-level. If $\mathcal{T}$ is of free type, \rref{eoper3}
induces an equivalence of monads
$$C^\prime \r C$$
where $C$ is the monad associated with $\mathcal{C}$, and $C^\prime$ is the monad
associated with $|\mathcal{C}^\prime|$; recall that for an operad $\mathcal{C}$, the
associated monad is
\beg{eoper4}{C(X)=\coprod_{n\geq 0} \mathcal{C}(n)\times_{\Sigma_n}X^n.
}
Thus, the rectification of the nerve $|X|$ of a weak $\mathcal{C}$-algebra $X$ from
a $|\mathcal{C}^\prime|$-algebra to a $\mathcal{C}$-algebra is
$$
\diagram
B(C,C^\prime, X) &\B(C^\prime, C^\prime, X)\lto_\sim \rto^\sim & X.
\enddiagram
$$
Now the exact same discussion applies to {\em multisorted algebras of operadic type}.
This means that there is an additional set $K$ of objects and each $i\in I$ comes with an 
$n_i$-tuple $(k_1,\dots, k_{n_i})\in K^{n_i}$ of input objects, and an output object
$\ell_i\in K$. An algebra of the type then consists of sets $X_k$, $k\in K$
and operations using elements of the prescribed input sets, and producing an
element of the prescribed output set. Again, the operations must satisfy
the prescribed relations, which are of type \rref{eoper1} (the only difference being
that every $x_i$ is decorated with an object and one must keep track of objects
when applying the operations). 

The associated operad \rref{eoper1a} then becomes a multisorted operad with objects $K$
(which is actually the same thing as a multicategory with objects $K$). The associated monad
to a multisorted operad is in the category of $K$-tuples of sets, which can also
be considered as sets fibered over $K$ (i.e. sets together with a map into $K$).
Formula \rref{eoper4} is then correct with $\times$ replaced by $\times_K$, the
fibered product over $K$.

The definition of free operadic type remains the same, with $\Sigma_n$ replaced
by the isotropy group of a particular fibration $\{1,\dots,n\}\r K$.

Now for us, the main point is that multicategories with object set $B$ (over $Id:B\r B$)
are a multisorted algebra of operadic type, where the set of objects is
$$\coprod_{n\geq 0} B^{n+1}.$$
Furthermore, this type is free (as is readily verified by inspection of the axioms). Hence,
weak multicategories can be rectified into strict multicategories using the May-Thomason 
rectification. 

\vspace{3mm}
\subsection{Rectification of weak multimorphisms}\label{ssrect}

A {\em weak morphism of weak $\mathcal{T}$-algebras} has the same data as 
a morphism (i.e. a map of sets preserving the operations), but instead for each operation,
we have a coherence isomorphism. The coherence diagrams in this case are easier:
there is one coherence diagram for each relation \rref{eoper1}.

The construction described in the previous subsection is functorial, so it automatically rectifies a weak
morphism to a morphism. (In the multi-sorted case, it even handles a function on objects.)
In particular, a weak
multifunctor between weak multicategories is rectified into a strict one. 

Suppose now 
$$F:\mathcal{A}\r \mathcal{C}$$
is a weak multifunctor where $\mathcal{A}$ is a weak multicategory and $\mathcal{C}$
is a strict multicategory enriched in groupoids. Then we should be entitled to more information
(i.e. we should not be required to rectify the already strict multicategory $\mathcal{C}$).
In effect, this works out. Using $C$ and $C^\prime$ in the same meaning as above,
we get a morphism of $C^\prime$-algebras
$$B_2F:B_2\mathcal{A}\r B_2\mathcal{C}.$$
Hence, we obtain morphisms of $C$-algebras (i.e. multifunctors)
$$\diagram
B(C,C^\prime,B_2\mathcal{A})\rrto^{B(C,C^\prime,B_2F)} &&
B(C,C^\prime, B_2\mathcal{C})
\dto^{B(C,\epsilon,B_2\mathcal{C})}\\&&B(C,C,B_2\mathcal{C})\dto\\
&&B_2\mathcal{C}
\enddiagram
$$
where the last map is the usual simplicial contraction.

\vspace{3mm}

\subsection{The topological Joyal-Street construction}

In this paper we have to deal with categories where both the sets of objects and
morphisms are topological spaces. The appropriate setting then is the notion of
a {\em T-category}, by which we mean a category $C$ where both the sets of objects
and morphisms are (compactly generated) topological spaces, $S,T:Mor(C)\r Ob(C)$
are fibrations and the unit and composition are continuous. In fact, for simplicity,
let us assume that $C$ is a groupoid, by which we mean
that there is an inverse operation $Mor(C)\r Mor(C)$ which is continuous. 
Then we shall speak of a $T$-groupoid.

Next, a {\em symmetric monoidal T-category} is a T-category $C$ with a continuous
functor $\oplus:C\times C\r C$ (continuous on both objects and morphisms)
and a unit $0\in Obj(C)$ satisfying the usual axioms of a symmetric monoidal
category, with the coherence natural transformations continuous (as maps
$Obj(C)\r Mor(C)$). A {\em permutative T-category} is a symmetric monoidal 
T-category where $\oplus$ is strictly associative unital.

\vspace{3mm}
\begin{proposition}\label{papp}
Let $C$ be a symmetric monoidal T-groupoid. Then there exists a 
permutative T-groupoid $C^\prime$ and a continuous weakly symmetric monoidal
functor (with continuous coherences)
$$\Gamma:C\r C^\prime$$
which is a continuous equivalence of categories $T$-categories (i.e. has a continuous
inverse where the compositions are continuously isomorphic to the identities).
\end{proposition}

\begin{proof}
Just as in the classical case, the proof is exactly the same as when we replace
``symmetric monoidal'' by ``monoidal'' and ``permutative'' by ``strictly associative
unital''. 

The T-category $C^\prime$ has objects which are continuous functors $E:C\r C$
together with continuous natural transformations
$$\diagram (E?)\oplus ?\rto^{\cong}& E(?\oplus ?). \enddiagram
$$
Morphisms are continuous natural isomorphisms $\diagram E\rto^\cong & E^\prime\enddiagram$
together with a commutative diagram
$$\diagram (E?)\oplus ?\dto_\cong
\rto^\cong & E(?\oplus ?)\dto^\cong\\
(E^\prime ?)\oplus ?\rto_\cong & E^\prime(?\oplus ?).
\enddiagram
$$
The functor $\Gamma:C\r C^\prime$ is (on objects)
$$X\mapsto X\oplus ?.$$
To prove that $C^\prime$ is a T-category and that $\Gamma$ is continuous, the key point
is to ge a more explicit description of the objects and morphisms of $C^\prime$.

We see that an object $E\in Obj(C^\prime)$ is determined by
the object $X=E(0)$ and a continuous choice of isomorphisms
$$\diagram X\oplus Y\rto^\cong & Z,\enddiagram$$
$$Y\in Obj(C).$$
Therefore,
$$Obj(C^\prime)=Obj(C)\times_{Map(Obj(C),Obj(C))}Map(Obj(C),Mor(C))$$
where 
$$Obj(C)\r Map(Obj(C),Obj(C))$$
is the adjoint to $\oplus$, and 
\beg{estreet*}{Map(Obj(C),Mor(C))\r Map(Obj(C),Obj(C))}
is $Map(Id,S)$. 
Similarly, morphisms $E\r E^\prime$ are determined by $E,E^\prime$ and a morphism
$$\diagram E(0)\rto^\cong & E^\prime(0).\enddiagram$$
Therefore, 
$$Mor(C^\prime)=Mor(C)\times_{Map(Obj(C),Obj(C))^2}Map(Obj(C),Mor(C))^2$$
where the Cartesian coordinates of
$$Mor(C)\r Map(Obj(C),Obj(C))^2$$
are adjoint to $S?\oplus ?$, $T?\oplus ?$. The key point of proving that $C^\prime$
is a (permutative) T-category is that \rref{estreet*} is a fibration, and hence so
are $S_{C^\prime}$, $T_{C^\prime}$. 

The continuous inverse of the functor $\Gamma$ is
$$E\mapsto E(0).$$
\end{proof}

\vspace{3mm}

\subsection{Singular vectors in Verma modules}

Consider the Verma module $V_{p,q}=V(c_{p,q},0)$ over the NS algebra (cf. \cite{gk}) where $c_{p,q}$
is given by \rref{ecpq}. It is known (\cite{ik1}) that $V_{p,q}$ has two singular vectors, one of which
is 
$$G_{-1/2}1$$
and the other, which we denote by $w$, has degree 
$$\frac{1}{2}{(p-1)(q-1)}.$$
In this subsection, we will compute some information about the singular vector $w$. Although
\cite{ik1} do compute a certain projection of $w$, their projection appears to annihilate the terms
we need, and we were not able to find another reference which would include them.

Denote 
$$V_{p,q}^{\prime}=V_{p,q}/\mathcal{A}\cdot (G_{-1/2}1)$$
where $\mathcal{A}$ is the NS algebra. Then $V_{p,q}^{\prime}$ has a basis consisting 
of vectors
\beg{esingv+}{G_{m_1}\dots G_{m_k}L_{n_1}\dots L_{n_\ell}}
where
$$\begin{array}{l}m_i\in \Z+\frac{1}{2}, n_j\in\Z,\\[2ex]
-\frac{3}{2}\geq m_1>m_2>\dots> m_k,
-2\geq n_1\geq n_2\geq\dots\geq n_\ell.
\end{array}$$
Let, for $(\alpha_0,\alpha_1,\alpha_2,\dots)$, $\alpha_i\in\N_0$,
$(\alpha_0,\alpha_1,\dots)>(\beta_0,\beta_1,...)$ if there exists an $i$ such that
$\alpha_j=\beta_j$ for $j<i$ and $\alpha_i>\beta_i$.
Let $I$ be the set of all sequences $(\alpha_0,\alpha_1,\dots)$ of non-negative integers
where for all but finitely $i$, $\alpha_i=0$.

We introduce an $I$-indexed increasing filtration on $V_{p,q}$ where
$$F_{(\alpha_0,\alpha_1,\dots)}V^{\prime}_{p,q}$$
is spanned by all elements \rref{esingv+} such that $(\beta_0,\beta_1,\dots)\leq (\alpha_0,
\alpha_1,\dots)$ where $\beta_0=k+\ell$ and $\beta_i$ is the number of $m_i$
(resp. $n_j$) equal to $-1-(i/2)$.

\vspace{3mm}

\begin{lemma}
\label{lsingv}
The projection $w^\prime\in V^{\prime}_{p,q}$ of $w$ is, up to non-zero multiple, equal to
\beg{esingv1}{G_{-5/2}G_{-3/2}L_{-2}\dots L_{-2}+\lambda L_{-2}\dots L_{-2}\;\;
\text{if $(p-1)(q-1)$ is even}
}
\beg{esingv2}{G_{-3/2}L_{-2}\dots L_{-2}\;\;\text{if $(p-1))(q-1)$ is odd}
}
plus elements of lower filtration degree, where $\lambda$ is an appropriate non-zero number.
\end{lemma}

\begin{proof}
Consider the highest filtration monomial summand $q$ of the form \rref{esingv+} of $w^\prime$. 
Assuming our statement is false, then the filtration degree of $q$ must be lower than
the filtration degree of \rref{esingv1} (resp. \rref{esingv2}). 

\vspace{3mm}
\noindent
{\bf Case 1:} $\beta_1=0$. Then let $i$ be the lowest such that $\beta_i>\alpha_i$
(this must exist by dimensional reasons). If $i-1$ is even, let
$$u=L_{(i-1)/2}w^\prime,$$
if $i-1$ is odd, let
$$u=G_{(i-1)/2}w^\prime.$$
In either case, the highest filtration degree of $u$ is $$(\beta_0,1,\beta_2,\dots, 
\beta_{i-1}, \beta_i-1,\beta_{i+1},\dots,...).$$
In particular, it cannot cancel with $L_{(i-1)/2}$ resp. $G_{(i-1)/2}$ being applied to 
lower filtration summands of $w^\prime$. This contradics $w^\prime$ being a singular vector.

\vspace{3mm}
\noindent
{\bf Case 2:} $\beta_1=1$. Let $i$ be the lowest such that $\beta_i>\alpha_i$. If $i-2$ is
even, let
$$u=L_{(i-1)/2}w^\prime,$$
if $i-2$ is odd, let
$$u=G_{(i-2)/2}w^\prime.$$
The rest of the argument is the same as in Case 1. Now in the case of $(p-1)(q-1)$ even,
note that
$$G_{1/2}(G_{-5/2}G_{-3/2}L_-2\dots L_{-2})$$
produces a summand of
$$G_{-3/2}L_{-2}\dots L_{-2},$$
which can cancel only with
$$G_{1/2}(L_{-2}\dots L_{-2}).$$
\end{proof}

\vspace{3mm}
\begin{theorem}
\label{tc2}
The quotient of $L_{p,q}$ by the sum of images of $a$ where $a$ are the coefficients
of $Y(u,z)$ at $z^{\geq 1}$ with $u\in L_{p,q}$ is generated by
\beg{ec2gen}{(L_{-2})^i, G_{-3/2}(L_{-2})^i}
with $0\leq i<(p-1)(q-1)/4$ if $p,q$ are odd, and $0\leq i<((p-1)(q-1)+1)/4$ if
$p,q$ are even (and $p-q\equiv 2\mod 4$). In particular, $L_{p,q}$ satisfies
the $C_2$ condition (cf. \cite{dlm}).
\end{theorem}

\begin{proof}
Similar to \cite{dlm}. Modulo lower filtration degrees, all elements \rref{esingv+} are
in the submodule $C_2L_{p,q}$ generated by $(coeff_{z^{\geq 1}}Y(?,z))L_{p,q}$
unless $n_\ell=-2$ (or $\ell=0$) and either $k=0$, or
$k=1$ and $m_1=-3/2$, or $k=2$ and $m_2=-5/2$. By Lemma \ref{lsingv}, and Lemma 3.8 of \cite{dlm}, then,
the listed elements generate the quotient $L_{p,q}/C_2L_{p,q}$.
\end{proof}

\end{document}